\documentclass[11pt,thmsa]{article}
\usepackage{amsmath, latexsym, amsfonts, amssymb, amsthm, amscd, url}
\usepackage{color}
\usepackage{amsmath,amssymb,amsthm, color}
\usepackage[english]{babel} 
\usepackage{babelbib}
\usepackage{version,tabularx,multicol}
\usepackage{graphicx,float, url}
\usepackage{amsfonts}
\usepackage{pdfsync}
\usepackage{color}
\usepackage{graphicx}
\usepackage{epsfig}
\usepackage{subfigure}
\usepackage{color}
\usepackage{float}
\usepackage{color}
\usepackage{tikz}
\usepackage{flafter}

\newtheorem{theorem}{Theorem}

\newtheorem{proposition}{Proposition}

\newtheorem{remark}{Remark}

\newtheorem{definition}{Definition}
\newtheorem{example}{Example}
\newcommand{\p}{\Bbb{P}}

\newcommand{\N}{\mbox{\rm I\hspace{-0.02in}N}}
\newcommand{\R}{\mathbb{R}}

\newcommand{\E}{\ensuremath{\mathbb{E}}}

\usepackage[T1]{fontenc}
\usepackage[utf8]{inputenc}
\usepackage{authblk}

\title{Seed bank Cannings graphs: How dormancy smoothes random genetic drift}

\author[1]{Adri\'an Gonz\'alez Casanova}
\author[2]{Lizbeth Pe\~naloza}
\author[3]{Arno Siri-J\'egousse}

\affil[1]{\footnotesize Unidad Cuernavaca del Instituto de Matem\'aticas de la  Universidad Nacional Aut\'onoma de M\'exico,
Av. Universidad s/n Periferica, 62210 Cuernavaca, Morelos, M\'exico. Email: adrian.gonzalez@im.unam.mx} 
\affil[2]{\footnotesize  Instituto de Investigaci\'on de Matem\'aticas y Actuar\'ia, Universidad del Mar, campus Huatulco. Carretera Federal No 200, Km 250, Santa Mar\'ia Huatulco 70989, Oaxaca, M\'exico. Email: lizbeth@huatulco.umar.mx}
\affil[3]{\footnotesize Instituto de Investigaciones en Matem\'aticas Aplicadas y Sistemas, Universidad Nacional Aut\'onoma de M\'exico, Circuito Escolar 3000, C.U., 04510 Coyoac\'an, CDMX, M\'exico. Email: arno@sigma.iimas.unam.mx}

\begin{document}

\maketitle

\begin{abstract}
In this article, we introduce a random (directed) graph model for the simultaneous forwards and backwards description of a rather broad class of Cannings models with a seed bank mechanism. This provides a simple tool to establish a sampling duality in the finite population size, and obtain a path-wise embedding of the forward frequency process and the backward ancestral process. Further, it allows the derivation of limit theorems that generalize celebrated results by M\"ohle to models with seed banks, and  where it can be seen how the effect of seed banks affects the genealogies. 
The explicit graphical construction is a new tool to understand the subtle interplay of seed banks, reproduction and genetic drift in population genetics.

\end{abstract}

\section{Introduction}
Cannings models and their modifications,  along with their multiple merger genealogies are a major topic in mathematical population genetics \cite{Sagitov1999, Sagitov_2001, Schweinsberg_2003, BHS18, Fre20, GMS21, SJW}.  Also, in the last decade, the study of dormancy {(also called seed bank effect)} received significant attention \cite{KKL, BGKS, BGKW, Blath_2020}. 
One of the unifying themes in both modeling areas is that they arise from extensions of the Wright-Fisher model, and that classical evolutionary forces such as genetic drift and selection are affected in important ways. While the theory of Cannings models is now robust, the study of models with dormancy is still work in progress. The main goal of this paper is to stabilize a framework in which seed banks can be combined with Cannings models, and to generalize the known limiting results for models without dormancy to Cannings models with seed bank.

An important tool in population genetics is the moment duality for Markov processes.
This technique establishes a mathematical relation between forward and backward in time processes.  The celebrated duality between the Wright-Fisher diffusion and the Kingman coalescent was gradually generalized to a wide class of neutral population genetics models, including some finite size discrete populations such as Cannings-type models \cite{Gonz_lez_Casanova_2018}.
In this situation, the duality leads to asymptotic results for both forward frequency and genealogical processes. In the context of dormancy, duality was established for the seed bank diffusion, which arises as a limit in models with geometric seed bank \cite{BGKW}. However, for discrete seed bank models, the duality relation is the first open gap that this article aims to close.
  
%

There are two main models for dormancy phenomena.
\begin{itemize}
\item{Kaj et al.}  The model defined in \cite{KKL} is based on the Wright-Fisher model with additional multi-generational jumps of (bounded) size, the system has been extended to geometric jump sizes of bounded expected range in \cite{Koopmann_2017} (which also provide some insight into the forward in time frequency diffusion), to the general finite expectation case in \cite{BGKS}, and even to unbounded (heavy-tailed) jump sizes in \cite{BEGK}.
\item{Blath et al.} A second modeling frame is given by an external  seed bank in terms of a ``second island''   (in the spirit of Wright's island model), effectively leading to geometric jump sizes on the evolutionary scale. Here, forward and backward limits have been constructed, giving rise to the seed bank diffusion and the seed bank coalescent \cite{BGKW} (see more analysis and generalization in \cite{GPS, Blath_2020} and an interesting connection with metapopulations in \cite{L2013}). 
\end{itemize}


Both modeling frames (generational jumps and second island) have their advantages and disadvantages. For the Wright-Fisher model with multi-generational jumps, one typically loses the Markov property. For the island version, one retains the Markov property but then needs to investigate two-dimensional frequency processes, which in the limit are harder to analyze than one-dimensional diffusions, since e.g. the Feller theory is missing (this can in part be replaced by recent theory for polynomial diffusions \cite{BBGW19}). Interestingly, it turns out that for the limiting frequency processes, both approaches can be two sides of the same medal.

In none of the above approaches, more general reproductive mechanisms, such as based on Cannings models, have been analyzed. This paper's second aim is to close this gap. 
We present an extended framework for the simultaneous construction of seed bank models with general multi-generational jump distributions and Cannings-type reproductive laws satisfying a paintbox construction. We are also able to obtain forward and backward convergence results (extending \cite{KKL}, \cite{Koopmann_2017} and \cite{BGKS}) and to provide an explicit sampling duality, which is valid already in the finite individual models.

More precisely, we show that if a sequence of Cannings models (with no seed bank effect) is in the universality class of the Kingman coalescent, meaning that its ancestral process converges in the evolutionary scale to the Kingman coalescent, then the ancestry of the same sequence with a seed bank effect will converge to the Kingman coalescent delayed by a constant $\beta^2$, where $\beta<\infty$ is the expected number of generations that separates an individual from its ancestor. This extends the results of \cite{KKL} and \cite{BGKS}. Convergence of the frequency process to the solution of the Wight-Fisher diffusion with the same delay is also proved. We go further and study how sequences of seed bank models with divergent expectations can make sequences of Cannings models that originally were not in the Kingman class, converge to the Kingman coalescent. This is achieved using the mixing time of some auxiliary Markov chains introduced in \cite{KKL}.
If instead of considering Cannings processes in the Kingman class we consider that their genealogy converges to a $\Xi$-coalescent, we show that their seed bank modification converges to a $\Xi^\beta$-coalescent. {Heuristically, the transformation $\Xi\to\Xi^\beta$ consists} in dividing by $\beta$ all the non-dust boxes in a $\Xi$ paintbox event to obtain a $\Xi^\beta$ paintbox event. Similar asymptotics are shown for the forward process. All those results are extended for models in the presence of mutations.

 
Note that the interplay of general reproduction and seed banks with other evolutionary forces can be subtle, and we provide a framework for its analysis (also regarding the real-time embedding of coalescent-based estimates, see e.g. \cite{Blath_2020}).

The paper is organized as follows. In section \ref{S2} we construct a random graph that allows us to embed the ancestry and the frequency processes of both Cannings and dormancy models simultaneously and study the duality relation of the processes forward and backward in time. Furthermore, we analyze the scaling limits of the ancestral process in presence of skewed reproduction mechanisms and dormancy.  We give conditions for convergence to the Kingman coalescent and study scenarios beyond this universality class, where we can describe how seed bank phenomena reduce the typical size coalescence events when combining seed banks with Cannings models that would, in absence of the seed bank component, converge to a $\Lambda$- or a $\Xi$- coalescent. Section \ref{ffp} uses the moment duality to formally prove convergence of the frequency process to a Wright-Fisher diffusion. This intuitively clear result was missing in the literature, probably since the lack of Markov property for the frequency process makes usual techniques fail. In section \ref{S4} we study a variant of the seed bank random graph  where mutations are added and we extend the results obtained in sections \ref{S2} and \ref{ffp}.

\section{A random graph version of the model of Kaj, Krone and Lascoux}\label{S2}

Consider a discrete-time haploid population of constant size $N\geq1$ at each generation.
The vertex set $V^N=\mathbb{Z}\times[N]$ represents the whole population. 
For each individual $v\in V^N$, denote by $g(v)$ its generation and by $\ell(v)$ its label so that $v=(g(v),\ell(v))$.
We denote the $g$-th generation of the population by $V^N_{g}:=\{v\in V^N:g(v)=g \}$.
Set a probability measure $\mathcal{W}^N$ on the exchangeable probability measures on $[N]$.
Let $\{\bar W^N_g\}_{g\in\mathbb{Z}}$ be a sequence of independent $\mathcal{W}^N$-distributed random variables 
with $\bar W^N_g=\{W_{v}^{N}\}_{v\in V^N_g}$.
Each variable $W_{v}^{N}$ gives the reproductive weight of the individual $v$ in the population graph.
This multinomial setting can be extended to some more general Cannings models (as in \cite{Sagitov_2001}) or non-exchangeable reproductive success (as in \cite{SJW}).
Also, consider a sequence $\{m_N\}_{N\geq1}$ of integers and
set a probability measure $\mu^N$ on $[m_N]$.  
Let $\{J^N_v\}_{v\in V^N}$ be a collection of independent $\mu^N$-distributed  random variables.
 The variable $J^N_v$ says  how many generations ago an individual $v$'s mother is living.
Finally, set a collection of random variables in $[N]$, $\{U^N_v\}_{v\in V^N}$ such that $U^N_v$ is the label of the mother of $v$.  Its conditional distribution is 
$$
\mathbb{P}(U^N_{v}=k|J^N_{v}=j, \{\bar W^N_g\}_{g\in\mathbb Z})=W_{(g(v)-j,k)}^{N}.
$$ 
\begin{definition}\label{SeedBankgraph} (The seed bank random di-graph)
Consider the random set of directed edges $$E^N=\{(v,(g(v)-J^N_v,U^N_v)), \,\text{ for all }v\in V^N\}.$$
The \textit{seed bank random di-graph} with parameters $N$, $\mathcal{W}^N$ and $\mu^N$ is given by $G^N:=(V^N,E^N)$.
\end{definition} 
Two classical examples are
\begin{itemize}
\item the Kaj, Krone and Lascoux (KKL) seed bank graph \cite{KKL}, in this case  $\mu^N$ has finite support $[m]$, i.e. $m_N=m$, and $\mathcal W^N=\delta_{(1/N,...,1/N)}$.
\item the Cannings model with parameter $\mathcal{W}^N$ \cite{Can1974,Can1975,Sagitov_2001}, in this case $\mu^N=\delta_1$. 
\end{itemize}

For every $u,v\in V^N$ we denote by $\delta(u,v)$ the distance of $u$ and $v$ in the graph $G^N$, i.e. the number of vertices in a path from $u$ to $v$ or from $v$ to $u$. Now let us define the ancestral process associated with this graph.

\begin{definition}[The ancestral process]\label{AncProc}
Fix a generation $g_0$ and $S_{g_0}$  consisting in a sample of individuals living between generation $g_0$ and  $g_0-m_N+1$, i.e. $S_{g_0}\subset\cup_{i=1}^{m_N}  V^N_{g_0+1-i}$.
  For every $g\geq 0$, let $\mathcal A^N_g$ be the set composed by the most recent ancestors of the individuals of $S_{g_0}$ that live at a generation $g_0-g'$ for some $g'\geq g$, that is 
  
$$\mathcal A^N_g=\{v\in \cup_{g'=g}^\infty V^N_{g_0-g'}: \exists u\in S_{g_0}\text{ such that }\delta(u,v)\leq \delta(u,v')\, \text{for all 
} v'\in  \cup_{g'=g}^\infty  V^N_{g_0-g'} \}.$$

Define, for all $ i\in [m_N]$,  
$$A^{N, i}_g=| \mathcal A^N_g\cap V^N_{g_0-g+1-i}|$$
and  $\bar A^{N}_g= (A^{N, 1}_g,\dots, A^{N, m_N}_g)$. 
We call $\{\bar A^{N}_g\}_{g\geq0}$ the ancestral process.
In the sequel, we consider the initial configuration $S_{g_0}(\bar n)$, for $\bar n=(n_1,\dots,n_{m_N})$,  such that $n_i\ge0$ individuals are uniformly sampled (with repetition) from generation $g_0+1-i$. We denote the law of the ancestral process of this sample by $\p_{\bar n}$.
See Figure \ref{FigAncestral} for an illustration.
\end{definition}

\begin{figure}
\centering
\begin{tikzpicture}
\draw[->,thick] (8.9,8) -- (5.2,7.1);
\draw[->,thick] (8.9,6) -- (5.2,6.9);
\draw[->,thick] (8.9,3) -- (5.2,4);
\draw[->,thick] (8.9,4) -- (7.2,4.9);
\draw[->,thick] (7,5) -- (5.2,4.1);
\draw[->,thick] (7,3) -- (5.2,3);
\draw[->,thick] (5,7) -- (3.2,7);
\draw[->,thick] (5,4) -- (3.2,6.9);
\draw[->,thick] (5,3) -- (1.2,5.9);
\draw[->,thick] (3,7) -- (1.2,6.1);

\draw[->,thick] (1,6) -- (-0.8,5.1);
\draw (-1,1) circle(0.2);
\draw (-1,2) circle(0.2);
\draw (-1,3) circle(0.2);
\draw (-1,4) circle(0.2); 
\draw[fill=lightgray] (-1,5) circle(0.2);
\draw (-1,6) circle(0.2); 
\draw (-1,7) circle(0.2);
\draw (-1,8) circle(0.2); 

\draw (1,1) circle(0.2);
\draw (1,2) circle(0.2);
\draw (1,3) circle(0.2);
\draw (1,4) circle(0.2); 
\draw (1,5) circle(0.2);
\draw[fill=lightgray] (1,6) circle(0.2); 
\draw (1,7) circle(0.2);
\draw (1,8) circle(0.2); 

\draw (3,1) circle(0.2);
\draw (3,2) circle(0.2);
\draw (3,3) circle(0.2);
\draw (3,4) circle(0.2); 
\draw (3,5) circle(0.2);
\draw (3,6) circle(0.2); 
\draw[fill=lightgray] (3,7) circle(0.2);
\draw (3,8) circle(0.2); 

\draw (5,1) circle(0.2);
\draw (5,2) circle(0.2);
\draw[fill=lightgray] (5,3) circle(0.2);
\draw[fill=lightgray] (5,4) circle(0.2); 
\draw (5,5) circle(0.2);
\draw (5,6) circle(0.2);
\draw[fill=lightgray] (5,7) circle(0.2);
\draw (5,8) circle(0.2); 

\draw (7,1) circle(0.2);
\draw (7,2) circle(0.2);
\draw[fill=gray] (7,3) circle(0.2); \node at (6.7,3.3) {$v_5$};
\draw (7,4) circle(0.2); 
\draw[fill=lightgray] (7,5) circle(0.2);
\draw (7,6) circle(0.2);
\draw (7,7) circle(0.2);
\draw (7,8) circle(0.2); 

\draw (9,1) circle(0.2); 
\draw (9,2) circle(0.2);
\draw[fill=gray] (9,3) circle(0.2); \node at (9.5,3) {$v_1$};
\draw[fill=gray] (9,4) circle(0.2); \node at (9.5,4) {$v_2$};
\draw (9,5) circle(0.2);
\draw[fill=gray] (9,6) circle(0.2);\node at (9.5,6) {$v_3$};
\draw (9,7) circle(0.2);
\draw[fill=gray] (9,8) circle(0.2); \node at (9.5,8) {$v_4$};
\node at (9,0.5) {0};
\node at (7,0.5) {-1};
\node at (5,0.5) {-2};
\node at (3,0.5) {-3};
\node at (1,0.5) {-4};
\node at (-1,0.5) {-5};

\end{tikzpicture}
\caption{In this case $N=8$ and $m_N=2$. The gray circles represent the members of $S_0=\{v_1,v_2,v_3,v_4, v_5\}$ where, for example, $v_2=(0,4)$ and $v_5=(-1,3)$. The light gray circles represent the ancestors of the sample. $\bar{A}^8_0=(4,1)$, $\bar{A}^8_1=(2,2)$, $\bar{A}^8_2=(3,0)$, $\bar{A}^8_3=(1,1)$,  $\bar{A}^8_4=(1,0)$, $\bar{A}^8_5=(1,0)$.} 
\label{FigAncestral}
\end{figure}
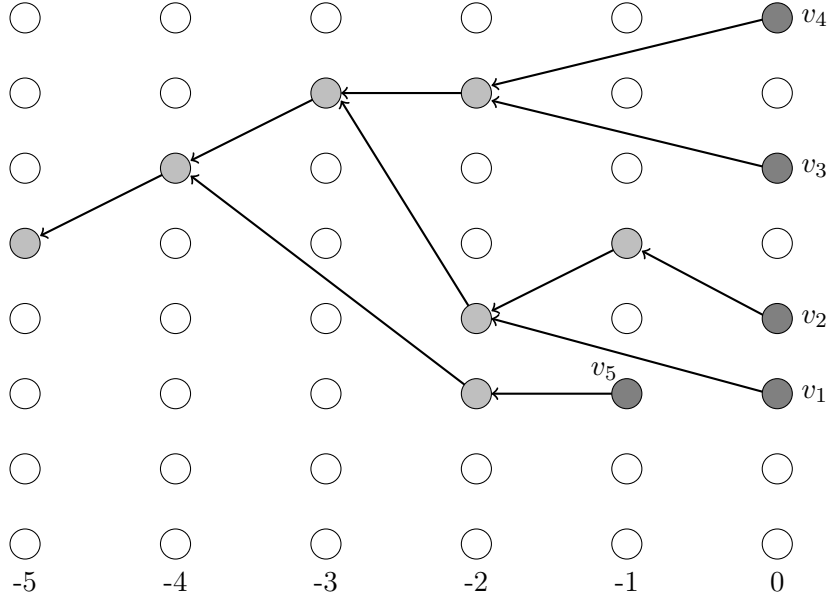
For simplicity, we suppose that $\sup\{i\ge1:n_i>0\}$ does not depend on $N$.
This model was introduced, for reproductions as in the Wright-Fisher model, by Kaj et al. \cite{KKL} directly, in the sense that they construct a random graph only implicitly. Our construction permits to provide a transparent relation between the ancestral process and the forward frequency process defined in section \ref{ffp}. 
Observe that $\{\bar A^{N}_g\}_{g\geq0}$ is a Markov chain. 
We start our results by formalizing the remark on p. 290 in \cite{KKL}.  This illustrative result is established when the Cannings model is in the domain of attraction of the Kingman coalescent, although it can be easily generalized to any type of reproduction law.
{Here we use the classical notations $c_N$ (resp. $d_N$) that denote the probability that two (resp. three) given individuals choose the same parent in a Cannings model. Those notations will be helpful all along the paper.
Recall, e.g. from \cite{Sagitov_2001}, that the genealogies of a Cannings model fall into the domain of attraction of the Kingman coalescent when $c_N\to0$ and $d_N=o(c_N)$ while multiple merger coalescents arise when $d_N$ and $c_N$ are of the same order. }

\begin{proposition}[Reformulation of Theorem 1 in \cite{KKL}]\label{transitionsA}
  Suppose that $c_N=N\E[( W^{N}_{v})^2]\to0$ and $d_N=N\E[( W^{N}_{v})^3]=o(c_N)$. Let $M(n)$ be a multinomial random variable with parameters $n$ and $\{\mu^N(i)\}_{i=1}^{m_N}$.  
  Also, for any $\bar n=(n_1,\dots, n_{m_N})\in [N]^{m_N},$ let $Z(\bar n) =(n_2,\dots, n_{m_N},0) +M(n_1)$.
Then, the transitions of $\{\bar A^{N}_g\}_{g\geq0}$ can be written in terms of $M$ and $Z$ as follows. 
\begin{itemize}
\item $\p_{\bar n}(\bar A^{N}_1=Z(\bar n))=1-\sum_{i=1}^\infty c_N[\binom{n_1}{2}\mu^N(i)^2+\mu^N(i)n_1n_{i+1}]+o(c_N)$ 
\item $\p_{\bar n}(\bar A^{N}_1=Z(\bar n)- e_i)=c_N[\binom{n_1}{2}\mu^N(i)^2+\mu^N(i)n_1n_{i+1}]+o(c_N)$
\end{itemize}
where $e_i$ is the vector with the $i$-th coordinate equal to $1$ and the others equal to $0$, for all $i\geq 1$.
\end{proposition}
\begin{proof}
We need to make two observations. First note that all the randomness in the transitions of the chain $\{\bar A^{N}_g\}_{g\geq0}$ lies in what happens to the first coordinate.  If for some $g\geq0$, $\bar A^{N}_g=(0,n_2,\dots, n_{m_N})$ it is easy to see that  $\bar A^{N}_{g+1}=(n_2,\dots, n_{m_N},0)$ almost surely. On the other hand, if $n_1>0$, the individuals that are in $\mathcal A^N_g\cap V^N_{g_0-g}$ cannot belong to $\mathcal A^N_{g+1}$. Then, each of these individuals, if denoted by $v$, must be replaced by an individual which lives $J^N_v$ generations in the past, that is 
$$\p_{e_1}(\bar A^{N}_{1}=e_i)=\mu^N(i).$$
Further, if $n_1>1$, one needs to find $n_1$ new ancestors, but some of them could be the same due to some coalescence.
The complete picture is as follows.
 For $i\geq2$ and $j,k\geq0$, and by denoting $e_0$ for the null vector,

\begin{equation}\label{coupling2}
\p_{2e_1+e_i}(\bar A^{N}_{1}=e_{i-1}+e_j+e_k) =\left\{ \begin{array}{ll}
2\mu^N(j)\mu^N(k) & \textrm{ if } i-1\neq j\neq k\\
(\mu^N(j))^2(1-c_N) & \textrm{ if } i-1\neq j, j=k\\
2\mu^N(i-1)\mu^N(k)(1-c_N) & \textrm{ if } i-1=j, j\neq k\\
(2\mu^N(i-1)\mu^N(j)+(\mu^N(j))^2)c_N & \textrm{ if } i-1\neq j,  k=0\\
(\mu^N(i-1))^2 (c_N-d_N) & \textrm{ if } i-1=j,  k=0\\
(\mu^N(i-1))^2 d_N& \textrm{ if } j=k=0
\end{array}
\right..
\end{equation}
The proof follows easily after these observations.    
\end{proof}

We now construct a less natural backward process which will be very useful when establishing its moment duality with the forward process in section \ref{ffp}. {We start by defining it in a graphical and intuitive way. More formal definitions will follow all along the section.}

{\begin{definition}[The window process]\label{WindProc}
Fix a generation $g_0$, and $S_{g_0} \subset\cup_{i=1}^{m_N}  V^N_{g_0+1-i}$.
In the genealogical tree of the sample $S_{g_0}$, define the variable $B_g^{N,1}$ as the number of edges {arriving to} generation $g_0-g$ (plus the number of individuals of $S_{g_0}$ living at this generation). For any $i\in\{2,m_N\}$, let $B_g^{N,i}$ be the number of edges crossing generation $g_0-g$ and {arriving to} generation $g_0-g-i+1$ (plus the number of individuals of $S_{g_0}$ living at this generation).
 Then, let
$
\bar B^{N}_g:= (B_g^{N,1},\dots,B_g^{N,m_N}).
$
We call $\{\bar B^{N}_g\}_{g\geq0}$ the window process.
As for the ancestral process, we denote by $\p_{\bar n}$ the law of the window process generated from the initial sample $S_{g_0}(\bar n)$. 
\end{definition}
In Figure \ref{FigAncestral}, the values of the window process are $\bar{B}^8_0=(4,1)$, $\bar{B}^8_1=(2,3)$, $\bar{B}^8_2=(5,0)$, $\bar{B}^8_3=(2,1)$,  $\bar{B}^8_4=(2,0)$, $\bar{B}^8_5=(1,0)$.}
The window process and the ancestral process only differ in the time where we acknowledge a coalescence event. In the window process coalescence events only occur in the first coordinate, while in the ancestral process coalescence events may take place at every entry (see Figure \ref{FigAncestral}).


The following equivalent (in law) definition of the window process allows us to compare it with the ancestral process.
 Let $C^N(n)$ be the number of ancestors after one generation of a sample of $n$ individuals in a Cannings model with weights distributed as $\mathcal{W}^N$. As in Proposition \ref{transitionsA}, let  $M(n)$ be a multinomial random variable with parameters $n$ and $\{\mu^N(i)\}_{i=1}^{m_N}$.  Given $\bar B^N_{g-1}=\bar n=(n_1,\dots, n_{m_N})\in [N]^{m_N}$, 
$$
\bar B^N_g=(n_2,\dots, n_{m_N},0)+ M(C^N(n_1)).
$$
 in distribution. It is left to the reader to show that indeed both definitions are equivalent.

The process  $\{\bar B_g^N\}_{g\geq0}$  can be expressed in terms of a particle system. 
Fix $N,\mu^N$ and $\mathcal{W}^N$. 
Let $Y^N_g=(R^N_g,L^N_g)$ define a Markov chain with state space $\N\times [N]$ and transition probabilities, conditional on the weights $\bar W^N_g$, $$
\p\big(Y^N_g=(i,k)|\{\bar{W}^N_{g}\}_g,Y^N_{g-1}=(1,j)\big)={W}_{(g_0-g+1-i,k)}^{N}\mu^N(i)
$$
for every $i\geq1$, $k\in [N]$, and 
$$
\p\big(Y^N_g=(i,k)|\{\bar{W}^N_{g}\}_g,Y^N_{g-1}=(i+1,j)\big)={W}_{(g_0-g,k)}^{N}
$$
for every $i\geq1$ and $k\in [N]$.
\begin{proposition}\label{ParticleEasy}
Set $n=\sum n_i$ to be the total size of the initial sample.
For every $g\geq0$, consider $n$ {(conditional on $\{\bar{W}^N_{g}\}_g$)} independent realizations  of $Y^N_g$, that we call $Y^{N,j}_g=(R^{N,j}_g,L^{N,j}_g)$ for $1\leq j\leq n$. Let $\sigma^{N,1}=\infty$ and $$\sigma^{N,j}=\inf\{g\geq 1: Y^{N,j}_g=Y^{N,j'}_g=(1,k),\text{ for some }j'<j \text{ such that }\sigma^{N,j'}>g, k\in[N]\}.$$  
For all $i\geq1$, set $\sum_{j=1}^n1_{\{R^{N,j}_0=i\}}=B^{N,i}_0$. Then,
 the $i$-th component $B_g^{N,i}$ of the random vector $\bar B_g^N$ is equal in distribution to
$\sum_{j=1}^n1_{\{R^{N,j}_g=i\}}1_{\{\sigma^{N,j}\ge g\}}$, for all $g\geq0$.
\end{proposition}
\begin{proof}
The proof consists in observing that {$g_0-R^{N,j}_g-g+1$} is equal in distribution  to the generation of the most recent ancestor, living at a generation $g_0-g'$ for some $g'\geq g$, of a fixed individual in the initial sample $S_{g_0}$. So we couple these two processes. 
At the particular times in which ${R^{N,j}_g=1}$ (and thus a coalescence event can occur {in the window process}) we take $L^{N,j}_g$ to be the label of the closest ancestor. 
Then $\sigma^{N,j}$ corresponds to the generation at which individual $j$'s ancestral lineage is involved into a coalescence event {with the ancestral lineage of an individual of lower level}. Under this coupling,
$$
B^{N,i}_g=\sum_{j=1}^n1_{\{R^{N,j}_g=i\}}1_{\{\sigma^{N,j}\ge g\}}
$$
almost surely.
\end{proof}

{At this point it could seem unnecessary that the process $L$ jumps at every time since it looks like it only has a role when $R$ is reaching 1. However this independent construction of $R$ and $L$ will be important in the proof of Theorem \ref{cv1KKL}.}
The chain $\{Y^N_g\}_{g\geq0}$ provides a very convenient coupling to the ancestral and the window processes, mainly because $\{R^N_g\}_{g\geq0}$ has an invariant measure given by 
\begin{equation}\label{nuN}
\nu^N(i)=\frac{\p(J^N_v\geq i)}{\E[J^N_v]}.
\end{equation} 
To see this, just observe that the chain has two types of behaviors. Using the notation $\p_j(\cdot)=\p(\cdot|R_{0}^N=j)$, we have
\begin{enumerate}
\item Deterministic transitions: if $j>1,$ then
 $\p_j(R_{1}^N={j-1})=1$
\item Random transitions:  for $j\geq1$,
 $\p_1(R_{1}^N={j})=\p(J^N_v=j)=\mu^N(j)$. 
\end{enumerate}
Then, 
\begin{eqnarray*}
\sum_{j=1}^\infty \p_j(R_{1}^N={i})\nu^N(j)&=&\p_{i+1}(R_{1}^N={i})\nu^N(i+1)+\p_1(R_{1}^N={i})\nu^N(1)\\&=&\frac{\p(J^N_v\geq i+1)}{\E[J^N_v]}+\frac{\p(J^N_v=i)}{\E[J^N_v]}\\
&=&\nu^N(i).
\end{eqnarray*}

{Our way to compare two Markov chains consists in applying coupling concepts developed in \cite{Mixing}.} Let us first recall the definition of mixing time (see page 55 of \cite{Mixing}).
 We denote $d^N(g)=\max_{j\in [m_N]}||\mathbb{P}_j(R^N_g\in\cdot)-\nu^N(\cdot)||_{TV}$ and the mixing time {$\tau_N=\inf\{g\geq0:d^N(g)<1/4\}$. We allow $\tau_N=0$, because this is true in the important case $\p(J^N_v=1)=1$, which is the case without seed bank. }


The main theorem of \cite{KKL} (proved for $\mu^N$ with finite support and  extended to finite expectation in \cite{BGKS}) consists in showing that the $L_1$  norm of the  ancestral process converges weakly to the block counting process of the Kingman coalescent under a constant time change. 
Here we extend this result to the window process and to some more general Cannings' mechanism.


\begin{theorem}[Convergence of the window process I: Kingman limit]\label{cv1KKL}
Consider a seed bank di-graph with parameters $N, \mathcal W^N$ and $\mu^N$.
Suppose that  $\beta_N=\E[J^N_v]<\infty$. Let $\tau_N$ be the mixing time of $\{R^{N}_g\}_{g\geq0}$, $c_N=N\E[( W^{N}_{v})^2]$ and $d_N=N\E[(W^{N}_{v})^3]$. Assume that  $\mu^N(1)>0$ and that
 \begin{equation}\label{cdcasoK}c_N/\beta_N^2 \rightarrow 0,\quad
N^\varepsilon\tau_Nc_N\rightarrow 0,\quad
(1/4)^{N^\varepsilon}\beta_N^2\rightarrow 0\quad\text{and}\quad
d_N/(\beta_Nc_N)\rightarrow 0.
\end{equation}
for some $\varepsilon>0$. Then, let $\{\bar B^N\}_{N\geq1}$ be the sequence of window processes with parameters $N$ and $\mu^N$ and starting condition $\bar B_{0}^N=\bar n$ for all $N\in\N$ big enough. Then,
\begin{equation}\label{conv}
\lim_{N\rightarrow \infty}\{|\bar B_{\lfloor t \beta_N^2/c_N\rfloor}^N|\}_{t\geq0}= \{N_t^K\}_{t\geq0}
\end{equation}
{in the finite dimension sense}, 
 where $\{N_t^{K}\}_{t\geq0}$ stands for the block counting process of a Kingman coalescent.
 
Furthermore, suppose that $\nu^N$ converges to a measure $\nu$ as $N\rightarrow \infty$. Let $V^{t,K}$ be a (conditional) multinomial random variable with parameters $N^K_t$ and $\nu$. For any fixed time $t>0$, in distribution,

\begin{equation}\label{multinomialdist}
\lim_{N\rightarrow \infty} \bar{B}_{\lfloor t \beta_N^2/c_N\rfloor}^N= V^{t,K}.
\end{equation}

\end{theorem}
Note that when $\beta_N\to\beta<\infty$ the third condition of Theorem \ref{cv1KKL} is automatically fulfilled. On the other side, when $\beta_N\rightarrow \infty$, then the fourth condition is always fulfilled because $d_N/c_N\leq 1$. The latter reflects the fact that a {stronger seed bank effect (in the sense that the expected number of generations separating an individual from its ancestor tends to infinity)} makes impossible the existence of multiple mergers. Theorem \ref{cv2KKL} below discusses the interplay between seed banks and random genetic drift.
{Note also that the second condition implies that $\tau_N<\infty$ for all but finitely many $N$, meaning that the support of $\mu^N$ is finite for all but at most finitely many $N$.} {In general, a process with finite mixing time can have an infinite support. In our case, a particle that starts in $k$ takes $k$ deterministic steps before jumping to a random location. This enforces that the support is a lower bound for the mixing time.}

 {
\begin{remark}
The arguments that we use in the proof of the theorem establish a convergence result in the finite dimension sense for \eqref{conv}. However this result can be strengthened into a convergence in distribution thanks to classical limit theorems, see for example {Theorem 17.25 in \cite{Kallenberg_2021}}. This will also be the case in the forthcoming Theorems \ref{cv2KKL}, \ref{Forward}, \ref{convergeKM}, \ref{cv2mutation} and \ref{Forwardmutation}.
\end{remark}}

{Recall from the seminal work of Schweinsberg \cite{Sch00} that $\Xi$-coalescents are characterized by a measure on $\Delta$, the infinite simplex on $[0,1]$. Every realization of this measure $\Xi$ describes the coalescence rule at every jump time of the process.} 
Let $F_\beta: \Delta\mapsto\Delta $ be such that for $A\subset\Delta$, $F_\beta(A)=\{ \bar y/\beta, \bar y\in A\}$. To any finite measure $\Xi$ over $\Delta$, we associate a finite measure $\Xi^\beta$ defined by the rule 
{$$\Xi^\beta(A):=\Xi(F_{1/\beta}(A)\cap \Delta)$$ for any borelian $A\subseteq \dfrac{1}{\beta}\Delta$, or equivalently
$$
\Xi^\beta(F_\beta(B))=\Xi(B)
$$
for any borelian $B$ on $\Delta$.}
Observe that $\Xi^\beta$ associates no weight on mass partitions $\bar y=(y_1,y_2,\dots)$ such that $\sum y_i>1/\beta$.
It is also remarkable that the seed bank effect sends the class of $\Lambda$-coalescents into itself. 
As an example, the $Beta(2-\alpha,\alpha)$ coalescent with the characteristic measure, restricted to $(0,1)$, $\Xi(A)=\int_{A}\frac{1}{\Gamma(2-\alpha)\Gamma(\alpha)}x^{1-\alpha}(1-x)^{\alpha-1}dx$ turns to a coalescent with measure, now restricted to $(0,1/\beta)$,
 $\Xi^\beta(A)=\int_{A}\frac{1}{\Gamma(2-\alpha)\Gamma(\alpha)}x^{1-\alpha}(\frac{1}{\beta}-x)^{\alpha-1}dx$.

\begin{theorem}[Convergence of the window process II: $\Xi$ limit]\label{cv2KKL}
Consider a seed bank di-graph with parameters $N, \mathcal W^N$ and $\mu^N$.
Suppose that  $\beta_N=\E[J^N_v]\to\beta <\infty$. Assume that the ancestral process of a Cannings model driven by $\mathcal{W}^N$, that we denote by $\{C^N_g\}_{g\geq0}$ is such that
\begin{equation}\label{cdcasoXi}
\lim_{N\rightarrow \infty}\{C^N_{\lfloor t/c_N \rfloor}\}_{t\geq0})= \{N_t^\Xi\}_{t\geq0}
\end{equation}
where $\{N_t^{\Xi}\}_{t\geq0}$ stands for the block counting process of a $\Xi$-coalescent.
Then 
\begin{equation}\label{convXi}
\lim_{N\rightarrow \infty}\{|\bar B_{\lfloor t/c_N\rfloor}^N|\}_{t\geq0}= \{N_{t}^{\Xi^\beta}\}_{t\geq0}
\end{equation}
in the finite dimension sense.

Furthermore, suppose that $\nu^N$ converges to a measure $\nu$ as $N\rightarrow \infty$. Let $V^{t,\Xi^\beta}$ be a (conditional) multinomial random variable with parameters $N_{t}^{\Xi^\beta}$ and $\nu$. For any fixed time $t>0$, in distribution, 
\begin{equation}\label{multinomialdistXi}
\lim_{N\rightarrow \infty} B_{\lfloor t/c_N\rfloor}^N= V^{t,\Xi^\beta}.
\end{equation}

\end{theorem}
{Note that the rates of the $\Xi^\beta$-coalescent converge to those of the Kingman coalescent when $\beta\to\infty$, converting the Kingman coalescent to a limit model when $\beta$ becomes large. This intuitively coincides with the hypothesis on $\beta$ in Theorem \ref{cv1KKL}.
}
\begin{remark}
The site frequency spectrum (SFS) provides a convenient way to visualize the transformations induced by the mapping $\Xi^\beta$  of Theorem \ref{cv2KKL} to the shape of the coalescent. The SFS is a vector proportional to the branch lengths $(L_1, \dots, L_{n-1})$ of the coalescent tree, where $L_i$ stands for the length of the lineages with exactly $i$ leaves in the original sample of size $n$.
In Figure \ref{SFS} we see how the branch lengths are increased in a non linear way in the case of the Bolthausen-Sznitman coalescent.
Note that when the genealogies are in the domain of attraction of the Kingman coalescent (Theorem \ref{cv1KKL}), every coordinate of the SFS will be multiplied by the constant $\beta$.
\end{remark}
\begin{figure}
\begin{center}
\subfigure{\includegraphics[width=6cm]{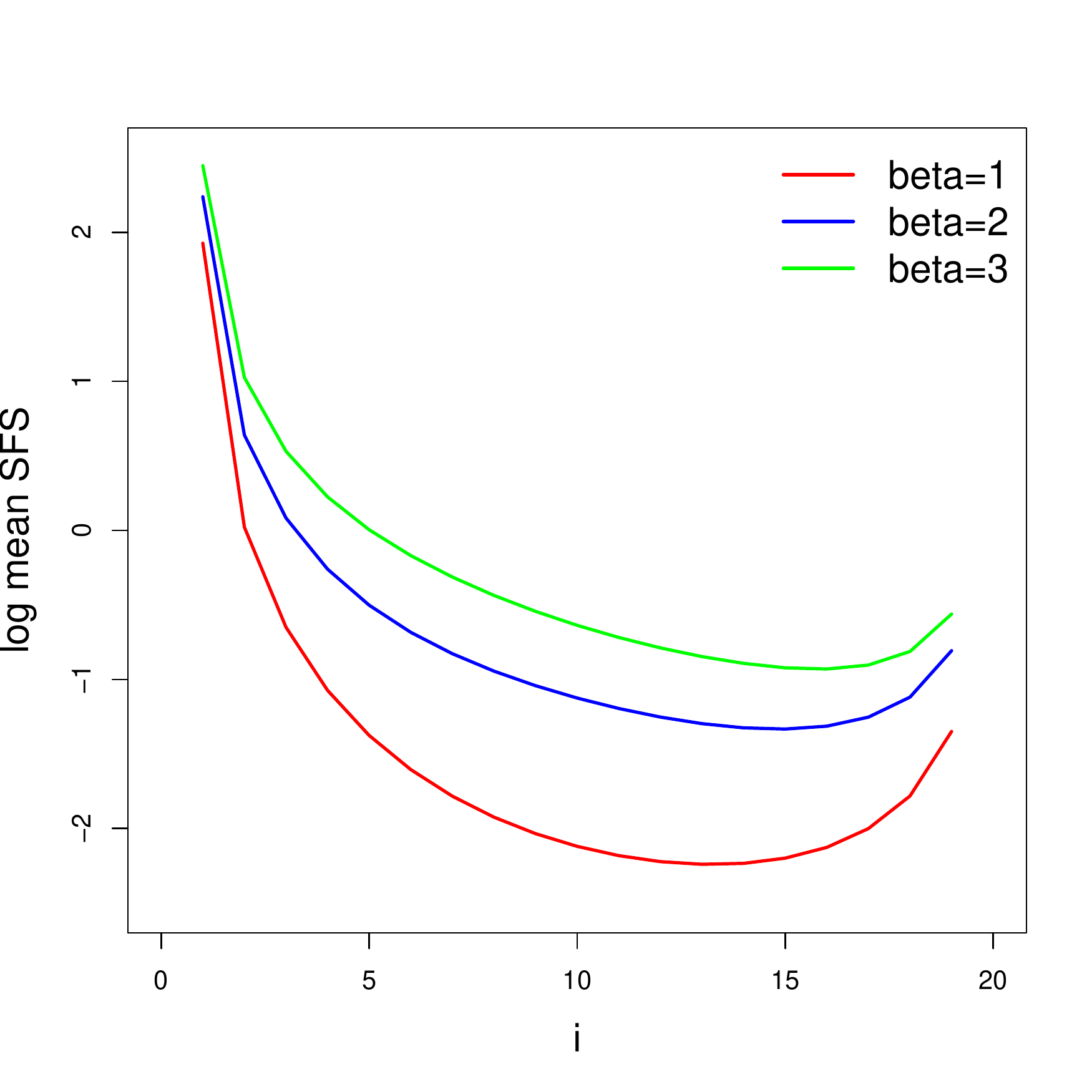}} \quad
\subfigure{\includegraphics[width=6cm]{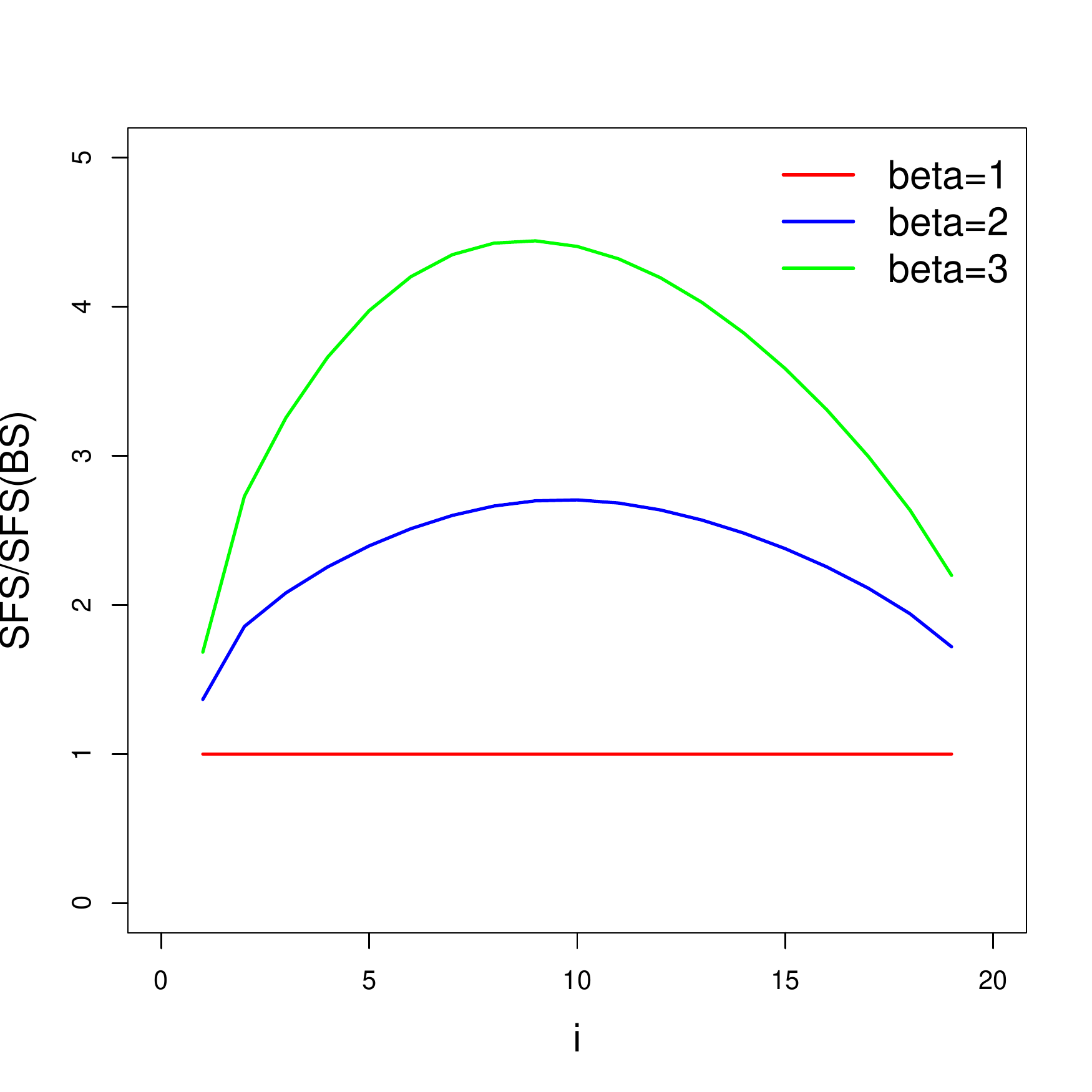}} 
\caption{Expected SFS of the $\Xi^\beta$-coalescent where $\Xi$ is the uniform measure on $(0,1)$ (Bolthausen-Sznitman coalescent) and $\beta=1,2,3$. 
The initial sample size is $n=20$.
The values are obtained thanks to computational methods based on phase-type theory \cite{HSJB}. An exact formula for $\beta=1$ can be found in \cite{KSJW}. 
Left: Expected SFS in logarithmic scale. Right: Expected SFS divided by the expected SFS of the Bolthausen-Sznitman coalescent.}\label{SFS}
\end{center}
\end{figure}

\begin{proof}[Proof of Theorem \ref{cv1KKL}]
The proof consists in coupling the window process $\{\bar B^N_g\}_{g\geq0}$ to a process which is "always in stationarity". 
If we suppose that $\{\bar B_{g}^N\}_{g\geq0}$ starts a.s. with one lineage, i.e. $\bar B_{0}^N=e_k$ for some $k$, it is straightforward that  it has a stationary distribution $\bar \nu^N$ given by $\bar \nu^N(e_i)=\nu^N(i)$.
Now, let $\underline{Y}^{N,j}_g=(\underline{R}^{N,j}_g,L^{N,j}_g)$ where $\{\underline{R}^{N,j}_g\}_{g\geq0}$ is a sequence of independent $\nu^N$-distributed random variables.  Let $$\underline{\sigma}^{N,j}=\inf\{g\geq1: \underline{Y}^{N,j}_g=\underline{Y}^{N,j'}_g=(1,k),\text{ for some }j'<j \text{ such that } \underline{\sigma}^{N,j'}>g, k\in[N]\}.$$
Hence, inspired by the coupling of Proposition \ref{ParticleEasy}, we define an artificial window process by
$$
\bar Z^N_g=(Z^{N,1}_g,\dots,Z^{N,m_N}_g)
$$
where, as $n=\sum n_i$,
$$
Z^{N,i}_g=\sum_{j=1}^n1_{\{\underline{R}^{N,j}_g=i\}}1_{\{\underline{\sigma}^{N,j}\geq g\}}.
$$
The process $\{|\bar Z^{N}_g|\}_{g\geq0}$ is Markovian. 

We now proceed in two steps to prove \eqref{conv}. First, we calculate the generator of $\{|\bar Z^{N}_g|\}_{g\geq0}$ in order to discover its scaling limit. Let $f:\N\rightarrow \R$ be a bounded function. 
Then 
\begin{align*}
\mathcal G^Nf(n)&=\E[f(|\bar Z^{N}_1|)-f(n)]=\p(|\bar Z^{N}_1|=n-1)[f(n-1)-f(n)]+O(\p(|\bar Z^{N}_1|=n-2))\\
&=\binom{n}{2}{c_N}{(\nu^N(1))^2}[f(n-1)-f(n)]+O({d_N}{(\nu^N(1))^3})\\
&=\frac{c_N}{\beta_N^2}\left[\binom{n}{2}[f(n-1)-f(n)]+O(\frac{d_N}{\beta_Nc_N})\right].
\end{align*}
 So we conclude that
\begin{equation}\label{convweak}
\lim_{N\to\infty}\{|\bar Z^{N}_{\lfloor\beta_N^2 t/c_N\rfloor}|\}_{t\geq0}= \{N^K_{t}\}_{t\geq0}
\end{equation}
in the finite dimension sense.

Second, let us couple $\{|\bar Z^{N}_{\lfloor\beta_N^2 t/c_N \rfloor}|\}_{t\geq0}$ and $\{|\bar B^{N}_{\lfloor \beta_N^2t/c_N\rfloor}|\}_{t\geq0}$ to show that the same limit holds for the rescaled window process. 
The coupling consists in constructing for every $i\geq1$ the random variable $(\underline{R}^{N,1}_{\rho_i},\dots,\underline{R}^{N,m_N}_{\rho_i})$ as the optimal coupling {(defined in Remark 4.8, of \cite{Mixing})} of $(R^{N,1}_{\rho_i},...,R^{N,m_N}_{\rho_i})$,  {which depends on the initial condition $\bar n\in [N]^{m_N}$,} and the stationary distribution $(\nu^N)^{\otimes m_N}$, 
where the times $\{\rho_i\}_{i\geq1}$ correspond to the times where the processes $\{|\bar Z^{N}_{\lfloor{\beta_N^2} t/c_N\rfloor}|\}_{t\geq0}$ and $\{|\bar B^{N}_{\lfloor \beta_N^2t/c_N \rfloor}|\}_{t\geq0}$ can jump.
More precisely, if we denote for any $p,q\in[n]$, $\rho^{N,p,q}_k=\inf\{g>\rho^{N,p,q}_{k-1}:L^{N,p}_g=L^{N,q}_g\}$ (with $\rho_0^{N,p,q}=0$), then $\rho_i=\inf\{g>\rho_{i-1}:g=\rho^{N,p,q}_k$ for some $p,q\in[n]$ {with $p\neq q$,} and some $k\in\N \}$ (with $\rho_0{=0}$).
Note that we do not precise the dependence on $N$ in the notation.
 In our case, the probability that the coupling is successful 
\begin{eqnarray*}
p_N&:=&\inf_{\bar n \in [N]^{m_N}} \p_{\bar n}((\underline{R}^{N,1}_{\rho_1},\dots,\underline{R}^{N,m_N}_{\rho_1})=(R^{N,1}_{\rho_1},\dots,R^{N,m_N}_{\rho_1}))\\
&=&1-\sup_{\bar n \in [N]^{m_N}}\p_{\bar n}((\underline{R}^{N,1}_{\rho_1},\dots,\underline{R}^{N,m_N}_{\rho_1})\neq(R^{N,1}_{\rho_1},\dots,R^{N,m_N}_{\rho_1}))\\
&=&1-\sup_{\bar n \in [N]^{m_N}}\| \p_{\bar n}((R^{N,1}_{\rho_1},\dots,R^{N,m_N}_{\rho_1})=\cdot)-(\nu^N)^{\otimes m_N}(\cdot)\|_{TV}
\end{eqnarray*}
where $\p_{\bar n}$ stands for the law of $\{R^{N,1}_{g},\dots,R^{N,m_N}_{g}\}_{g\geq0}$ (or $\{\underline R^{N,1}_{g},\dots,\underline R^{N,m_N}_{g}\}_{g\geq0}$)  starting at the state $\bar n\in [N]^{m_N}$ and where Proposition 4.7 in \cite{Mixing} is used for the last equality.
To prove that $p_N\to1$ when $N\to\infty$, take $\varepsilon>0$ such that $N^\varepsilon\tau_N c_N\rightarrow 0$. 
The condition $\mu^N(1)>0$ implies that, for any $i\geq1$,  the processes $\{R_g^{N,i}\}_{g\geq0}$ are irreducible. So, by {equation (4.29)  in \cite{Mixing} and the definition of $\tau_N$}, we have
 
\begin{equation}\label{coumpl}
||\p_{\bar n}((R^{N,1}_{N^\varepsilon\tau_N },\dots,R^{N,m_N}_{N^\varepsilon\tau_N })=\cdot)-(\nu^N)^{\otimes m_N}(\cdot)||_{TV}<(1/4)^{N^\varepsilon}.
\end{equation}
Then observe that, stochastically, $\rho_1\geq \Gamma^N$ where $\Gamma^N$ is a geometric random variable with parameter $n^2c_N\geq\binom n2c_N$ {(think of a Cannings model without dormancy)} and thus {$$\p(\rho_1<N^\varepsilon\tau_N )\leq \p(\Gamma^N<N^\varepsilon\tau_N)=1-(1-n^2c_N)^{N^\varepsilon\tau_N}\rightarrow 0.$$}
{This implies that $p_N\ge1-(1/4)^{N^\varepsilon}$.}

 Let $T^N_1=\inf\{i\geq1: |\bar Z^{N}_{\rho_i}|=1\}$. 
Observe that if $|\bar Z^{N}_{0}|=n$ (we will use the notation $\p_n$), stochastically 
\begin{equation}\label{trucmuche}T^N_1\leq \sum_{i=1}^nG^N_i \end{equation}
where the $G^N_i$'s are independent geometric random variables with parameter $\beta_N^{-2}$ {(the probability that two particles in stationarity reach the origin)}. 
We finish the proof of  \eqref{conv} noting that the trajectories of both processes are identical with overwhelming probability
{\begin{eqnarray*}\label{trajectories}
\p_n(\sup_t||\bar Z^{N}_{\lfloor t\beta_N^2/c_N\rfloor}|-|\bar B^{N}_{\lfloor t\beta_N^2/c_N\rfloor}||=0)
&=&\E[p_N^{T^N_1}]\ge\E[p_N^{\sum_{i=1}^nG^N_i}]\nonumber\\
&=&\left(\frac{p_N\beta_N^{-2}}{1-(1-\beta_N^{-2})p_N}\right)^n\nonumber\\
&=&\left(1-\beta_N^2+p_N^{-1}\beta_N^2\right)^{-n}\nonumber\\
&\ge&\left(1-\beta_N^2+(1-(1/4)^{N^\varepsilon})^{-1}\beta_N^2\right)^{-n}\nonumber
\end{eqnarray*}  
and this quantity converges to 1 when $N\to\infty$.}

Finally, let us  prove  \eqref{multinomialdist}. Let $t>0$ fixed, and suppose that $\nu^N$ converges to a measure $\nu$. By, equation \eqref{convweak} we have that $\lim_{N\rightarrow\infty}|\bar{Z}^N_{\lfloor t\beta_N^2/c_N\rfloor}|=N_t^K$. Then, observe that $\bar{Z}_{g}^N$ has a multinomial distribution with parameters $|\bar{Z}^N_{g}|$ and $\nu^N$. Thus, in distribution,
\begin{equation}\label{convergeciaZvector}
\lim_{N\rightarrow\infty}\bar{Z}^N_{\lfloor t\beta_N^2/c_N\rfloor}=V^{t,K}. 
\end{equation}
On the other hand, by \eqref{conv}, we have that 
\begin{equation*}
\lim_{N\rightarrow\infty}|\bar{B}^N_{\lfloor t\beta^2_N/c_N\rfloor}|=N_t^K.
\end{equation*}

For $t>0$ fixed, let us couple $\bar{Z}^N_{\lfloor t\beta^2_N/c_N\rfloor}$ and $\bar{B}^N_{\lfloor t\beta^2_N/c_N\rfloor}$ to show that the limit \eqref{convergeciaZvector} is the same for $\bar{B}^N_{\lfloor t\beta^2_N/c_N\rfloor}$. As we did before, the coupling consists in constructing the random variable $(\underline{R}^{N,1}_{\lfloor t\beta^2_N/c_N\rfloor},\dots,\underline{R}^{N,m_N}_{\lfloor t\beta^2_N/c_N\rfloor})$ as the optimal coupling of $(R^{N,1}_{\lfloor t\beta^2_N/c_N\rfloor},\dots,R^{N,m_N}_{\lfloor t\beta^2_N/c_N\rfloor})$, { which depends on the initial condition $\bar n\in [N]^{m_N}$}. The probability that the coupling is successful is

\begin{eqnarray*}
\varrho_N&:=&\inf_{\bar n \in [N]^{m_N}} \p_{\bar n}((\underline{R}^{N,1}_{\lfloor t\beta^2_N/c_N\rfloor},\dots,\underline{R}^{N,m_N}_{\lfloor t\beta^2_N/c_N\rfloor})=(R^{N,1}_{\lfloor t\beta^2_N/c_N\rfloor},\dots,R^{N,m_N}_{\lfloor t\beta^2_N/c_N\rfloor}))\\
&=&1-\sup_{\bar n \in [N]^{m_N}}\p_{\bar n}((\underline{R}^{N,1}_{\lfloor t\beta^2_N/c_N\rfloor},\dots,\underline{R}^{N,m_N}_{\lfloor t\beta^2_N/c_N\rfloor})\neq(R^{N,1}_{\lfloor t\beta^2_N/c_N\rfloor},\dots,R^{N,m_N}_{\lfloor t\beta^2_N/c_N\rfloor}))\\
&=&1-\sup_{\bar n \in [N]^{m_N}}\| \p_{\bar n}((R^{N,1}_{\lfloor t\beta^2_N/c_N\rfloor},\dots,R^{N,m_N}_{\lfloor t\beta^2_N/c_N\rfloor})=\cdot)-(\nu^N)^{\otimes m_N}(\cdot)\|_{TV}
\end{eqnarray*}
Take $\varepsilon>0$ such that $N^\varepsilon\tau_N c_N\rightarrow 0$, and observe that $\beta_N>0$. This implies that
\begin{equation}\label{gamma}
\p(\lfloor t\beta^2_N/c_N\rfloor<N^\varepsilon\tau_N )\rightarrow 0 \mbox{ as } N\rightarrow\infty.
\end{equation}
By, \eqref{coumpl} and \eqref{gamma}, we have that $\varrho_N\to 1$. 
This gives \eqref{multinomialdist}.
\end{proof}

\begin{remark}

Consider two processes, $\{\underline{ R}^N_g\}_{g\geq0}$ and $\{R^N_g\}_{g\geq0}$, the first one starting with one particle in stationarity and the second one starting with one particle in state one.  If we consider their Doebling coupling (which consists in letting them evolve according to their respective laws and in merging their paths when they meet for the first time, see \cite{Mixing}, Chapter 5), their coupling time, $T^N$, is less than two with probability 
$$\nu^N(1)+\sum_{i=1}^{m_N-1}\mu^N(i)\nu^N(i+1)\leq \nu^N(1)+\frac{1}{\beta_N}\sum_{i=1}^{m_N-1}\mu^N(i)=\frac{1}{\beta_N}(2-\mu^N(m_N)).$$
As the process $\{R^N_g\}_{g\geq0}$ visits the state one approximately every $\beta_N$ steps we conclude that $\mathbb{P}(T^N>N^\varepsilon \beta_N^2)\rightarrow 0$ when $N\rightarrow \infty$. Since $\tau_N\leq\inf\{t\geq0;\mathbb P(T^N>t)<1/4\}$, we obtain that
 $m_N\le\tau_N\leq {N^\varepsilon\beta^2_N}$ {for $N$ large enough. Note that  $m_N$ is always a lower bound for $\tau_N$ since the mixing time is at least the time for one particle starting at $m_N$ to reach the origin.} Then hypotheses of Theorem \ref{cv1KKL} can be relaxed to the following
 $$c_N/\beta_N^2 \rightarrow 0,\quad{m_N/(N^{\varepsilon} \beta^{2}_N)\to\delta<1,\quad N^{2\varepsilon }\beta^2_Nc_N\rightarrow 0},\quad (1/4)^{N^\varepsilon}\beta_N^2\rightarrow 0\quad\text{and}\quad d_N/(\beta_Nc_N)\rightarrow 0$$
 with the advantage that they are easier to verify.

\end{remark}

\begin{proof}[Proof of Theorem \ref{cv2KKL}]
The proof of \eqref{convXi} is similar to that of \eqref{conv} in Theorem \ref{cv1KKL}.
First observe that \eqref{cdcasoXi} implies that $c_N\to0$.
Assuming furthermore that $\beta_N\to\beta<\infty$, we get that
$c_N/\beta_N^2\to0$, $
(1/4)^{N^\varepsilon}\beta_N^2\rightarrow 0$ and $N^\varepsilon\tau_Nc_N\rightarrow 0$ (since the mixing time is of order 1 in this case).
These are the three first conditions of \eqref{cdcasoK}, necessary to mimic the proof of Theorem \ref{cv1KKL}.

 In the present case, let $I_i$ denote the indicator of the event that $L_1^{N,i}=L_1^{N,j}$ for some $j\in[i-1]$. Note that $\{C^N_{\lfloor t/c_N \rfloor}\}_{t\geq0}$ has generator
$$
\mathcal{C}^Nf(n)=c_N^{-1}\E[f(n-\sum_{i=1}^nI_i)-f(n)]
$$
which by hypothesis converges to the generator of the block counting process of a $\Xi$-coalescent. Finally note that, using the same notation, the generator of the artificial (in stationarity) block counting process $\{\bar Z_{\lfloor t/c_N\rfloor}^N\}_{t\geq0}$ is 

$$
\mathcal{C}^Nf(n)=c_N^{-1}\E[f(n-\sum_{i=1}^nI_i1_{\{ \underline R^{N,i}_1=1\}})-f(n)].
$$
As $1_{\{\underline{R}_1^{N,i}=1\}}$ is a Bernoulli random variable with parameter tending to $\beta^{-1}$ and independent of $I_i$, we conclude that
$$
\lim_{N\to\infty}\{|\bar B_{\lfloor t/c_N\rfloor}^N|\}_{t\geq0}= \{N_{t}^{\Xi^\beta}\}_{t\geq0}.
$$
The rest of the proof is identical to that of Theorem \ref{cv1KKL}. 
\end{proof}

{
\begin{remark}
The window process can also be defined in terms of a modification of the partition-valued ancestral process. This more formal (and more complicated) definition follows the historical approach of \cite{KKL, BGKS}.
Fix a generation $g_0$, and $S_{g_0} \subset\cup_{i=1}^{m_N}  V^N_{g_0+1-i}$.
 For $g\geq1$, consider the equivalence relation on $S_{g_0}$, that we denote by $\sim_g$, such that $u\sim_g v$ if and only if they have a common ancestor at a generation between $g_0-g+1$ and $g_0$. 
 Let $\pi_0=\pi_1$ be the trivial partition made of the isolated elements of $S_{g_0}$ (singletons) and let $\pi_g$ be the partition induced by $\sim_g$ in the sample $S_{g_0}$. Let $\mathcal B^N_g$ be the set composed by the closest ancestors, living at a generation $g_0-g'$ for some $g'\geq g$, of each of the blocks in  $\pi_g$. Then, for $1\leq i\leq m_N$, we define
$$
B_g^{N,i}:=|\mathcal B^N_g\cap V^N_{g_0-g+1-i}|.
$$
We illustrate this definition by the realization pictured in Figure \ref{FigAncestral}. 

In this case $\pi_0=\pi_1$ $=\pi_2=\{\{v_1\},\{v_2\},\{v_3\},\{v_4\},\{v_5\}\}$. 
Observe that even if some individuals reach their common ancestor at generation -2, they  remain isolated in $\pi_2$.
Then $\pi_3=\{\{v_1,v_2\},\{v_3,v_4\},\{v_5\}\}$, $\pi_4=\{\{v_1,v_2,v_3,v_4\},\{v_5\}\}$, $\pi_5=\{\{v_1,v_2,v_3,v_4,v_5\}\}$.
Hence, $\mathcal B^8_0=\{v_1,v_2,v_3,v_4,v_5\}$
and, when moving some generations backwards, we get $\mathcal B^8_1=\{v_5,(-1,5),(-2,4),(-2,7),(-2,7)\}$ and $\mathcal B^8_2=\{(-2,3),(-2,4),(-2,4),(-2,7),(-2,7)\}$. Observe that in $\mathcal B^8_2$ the ancestors $(-2,4)$ and $(-2,7)$ appear twice. 
Also  $\mathcal B^8_3=\{(-3,7),(-3,7),(-4,6)\}$,  $\mathcal B^8_4=\{(-4,6), (-4,6)\}$,  $\mathcal B^8_5=\{(-5,5)\}$. 
\\
Finally, note that the coupled variables $Y^{N,j}_g=(R^{N,j}_g,L^{N,j}_g)$ define the process that models the distance between $g$ and the position of the ancestor of the $j$-th block induced by $\sim_g$.
\end{remark}}

\section{The forward frequency process}\label{ffp}

In this section we introduce the forward frequency process associated to the  seed bank graph, we establish duality results with the ancestral and window processes introduced in the previous section, and we establish some scaling limits results thanks to these tools.
\begin{definition}[The frequency process]\label{FreqProc}
Fix a generation $g_0$ and an initial sample $S_{g_0}\subset\cup_{i=1}^{m_N}  V^N_{g_0+1-i}$, that we call the type $A$ individuals. Hence, $\cup_{i=1}^{m_N}  V^N_{g_0+1-i}\backslash S_{g_0}$  is the set of type $a$ individuals. 
 For $g\geq0$, set (omitting again the dependence to $S_{g_0}$) 
$$
X^{N,i}_g=\frac{1}{N}|\{v\in V^N_{g_0+g+1-i}: v\text{ is not connected to }u\text{ for some }u\in \cup_{i=1}^{m_N}  V^N_{g_0+1-i}\backslash S_{g_0} \}|.
$$
Then, define the process of the neutral frequency of type $A$ individuals $\{\bar X^{N}_g\}_{g\geq0}$, by
$$
\bar X^{N}_g=(X_g^{N,1},\dots,X_g^{N,m_N}).
$$
Set a vector $\bar x=(x_1, \dots, x_{m_N})\in([N]/N)^{m_N}$. In the sequel, we suppose that the forward frequency process starts from a fraction $x_1$ of generation 0, a fraction $x_2$ of generation $-1$, and so on... We denote this sample by $S_{0}(\bar x)=\cup_{i=1}^{m_N}\cup_{k=1}^{ x_iN}\{(1-i,k)\}$ and we denote the law of the frequency process starting from this configuration  by ${\bf P}_{\bar x}$.
\end{definition}
\begin{figure}[H]
\centering
\begin{tikzpicture}
\draw[->,thick] (-0.8,2) to [out=-30, in=190](4.8,2);
\draw[->,thick] (-0.8,3) -- (0.8,3);
\draw[->,thick] (-0.8,8) to [out=60, in=160](2.8,8);
\draw[->,thick] (-0.8,1) to (0.8,2);
\draw[->,thick] (-0.8,4) -- (0.8,4.9);
\draw[->,thick] (-0.8,5) -- (2.8,6);
\draw[->,thick] (-0.8,5) -- (2.8,6);
\draw[->,thick] (-0.8,6) -- (.8,7);
\draw[->,thick] (-0.8,7) -- (.8,8);

\draw[->,thick] (1,1) -- (2.8,1);
\draw[->,thick] (1,4) -- (4.8,3);
\draw[->,thick] (1.2,5) -- (2.8,5);
\draw[->,thick] (1.2,6) -- (2.8,7);
\draw[->,thick] (1.2,8)to [out=30, in=160] (4.8,8);
\draw[->,thick] (1.2,3) -- (2.8,4);
\draw[->,thick] (1.2,3) -- (2.8,3);
\draw[->,thick] (1.2,2) -- (2.8,2);

\draw[->,thick] (3,1) -- (4.8,1);
\draw[->,thick] (3.2,2)to [out=30, in=160] (6.8,2);
\draw[->,thick] (3.2,3)to [out=30, in=200] (6.8,4);
\draw[->,thick] (3.2,4) -- (4.8,4);
\draw[->,thick] (3.2,5) -- (4.8,5);
\draw[->,thick] (3.2,5) -- (4.8,6);
\draw[->,thick] (3.2,6) -- (4.8,7);
\draw[->,thick] (3.2,7) to [out=30, in=200] (8.8,8);
\draw[->,thick] (3.2,8) to [out=30, in=150] (6.8,8);

\draw[->,thick] (5,1) -- (8.8,2);
\draw[->,thick] (5.2,2) -- (6.8,1);
\draw[->,thick] (5,3) -- (6.8,3);
\draw[->,thick] (5,3) -- (6.8,4.9);
\draw[->,thick] (5.2,4)to [out=30, in=150] (8.8,4);
\draw[->,thick] (5.2,5) -- (6.8,5.9);
\draw[->,thick] (5.2,6.1) -- (6.8,7);
\draw[->,thick] (5.2,8)to [out=30, in=200] (9,8.5);

\draw[->,thick] (7.2,1) -- (8.8,1);
\draw[->,thick] (7,3) -- (8.8,3);
\draw[->,thick] (7,5) -- (8.8,5);
\draw[->,thick] (7.2,6) -- (8.8,6);
\draw[->,thick] (7.2,8) -- (8.8,7);

\draw[fill=gray] (-1,1) circle(0.2); \node at (-1.5,1) {$v_1$};
\draw[fill=gray] (-1,2) circle(0.2);\node at (-1.5,2) {$v_2$};
\draw[fill=gray] (-1,3) circle(0.2);\node at (-1.5,3) {$v_3$};
\draw[fill=gray] (-1,4) circle(0.2); \node at (-1.5,4) {$v_4$};
\draw (-1,5) circle(0.2);
\draw (-1,6) circle(0.2); 
\draw (-1,7) circle(0.2);
\draw (-1,8) circle(0.2); 

\draw[fill=gray] (1,1) circle(0.2);\node at (.5,1) {$v_5$};
\draw[fill=gray] (1,2) circle(0.2);\node at (.5,2.2) {$v_6$};
\draw[fill=gray] (1,3) circle(0.2);\node at (.5,3.2) {$v_7$};
\draw[fill=gray]  (1,4) circle(0.2); \node at (0.5,4) {$v_8$};
\draw[fill=lightgray]  (1,5) circle(0.2);
\draw (1,6) circle(0.2); 
\draw (1,7) circle(0.2);
\draw (1,8) circle(0.2); 

\draw[fill=lightgray] (3,1) circle(0.2);
\draw[fill=lightgray]  (3,2) circle(0.2);
\draw[fill=lightgray]  (3,3) circle(0.2);
\draw[fill=lightgray]  (3,4) circle(0.2); 
\draw[fill=lightgray]  (3,5) circle(0.2);
\draw (3,6) circle(0.2); 
\draw (3,7) circle(0.2);
\draw (3,8) circle(0.2); 

\draw[fill=lightgray]  (5,1) circle(0.2);
\draw[fill=lightgray] (5,2) circle(0.2);
\draw[fill=lightgray] (5,3) circle(0.2);
\draw[fill=lightgray] (5,4) circle(0.2); 
\draw[fill=lightgray] (5,5) circle(0.2);
\draw[fill=lightgray] (5,6) circle(0.2);
\draw (5,7) circle(0.2);
\draw (5,8) circle(0.2);
 
\draw[fill=lightgray] (7,1) circle(0.2);
\draw[fill=lightgray] (7,2) circle(0.2);
\draw[fill=lightgray] (7,3) circle(0.2);
\draw[fill=lightgray] (7,4) circle(0.2); 
\draw[fill=lightgray]  (7,5) circle(0.2);
\draw[fill=lightgray] (7,6) circle(0.2);
\draw[fill=lightgray] (7,7) circle(0.2);
\draw (7,8) circle(0.2);
 
\draw[fill=lightgray] (9,1) circle(0.2); 
\draw[fill=lightgray] (9,2) circle(0.2);
\draw[fill=lightgray] (9,3) circle(0.2);
\draw[fill=lightgray] (9,4) circle(0.2); 
\draw[fill=lightgray]  (9,5) circle(0.2);
\draw[fill=lightgray] (9,6) circle(0.2);
\draw (9,7) circle(0.2);
\draw (9,8) circle(0.2); 
\node at (9,0.5) {4};
\node at (7,0.5) {3};
\node at (5,0.5) {2};
\node at (3,0.5) {1};
\node at (1,0.5) {0};
\node at (-1,0.5) {-1};
\node at (11.5,7.5) {$\bar{X}^8_0=\left(  \dfrac{5}{8},\dfrac{4}{8},0\right) $};
\node at (11.5,6) {$\bar{X}^8_1=\left(  \dfrac{5}{8},\dfrac{5}{8},\dfrac{4}{8} \right)$};
\node at (11.5,4.5) {$\bar{X}^8_2=\left(  \dfrac{6}{8},\dfrac{5}{8}, \dfrac{5}{8}\right)$};
\node at (11.5,3) {$\bar{X}^8_3=\left(  \dfrac{7}{8},\dfrac{6}{8}, \dfrac{5}{8}\right)$};
\node at (11.5,1.5) {$\bar{X}^8_4=\left(  \dfrac{6}{8},\dfrac{7}{8}, \dfrac{6}{8}\right)$};
\end{tikzpicture}
\caption{In this case $N=8$, $m_N=3$ and $\bar{x}=\left( \frac{4}{8},\frac{4}{8},0\right)$. The gray circles represent the members of $S_0(\bar{x})=\{v_1,v_2,v_3,v_4,v_5,v_6,v_7, v_8\}$ where, for example, $v_3=(-1,3)$ and $v_6=(0,2)$. The light gray circles represent the sample's offspring. It is useful to observe that $X_g^{N,i}=X_{g+1}^{N,i+1}$ for $i<m_N$.}
\label{Fig2}
\end{figure}
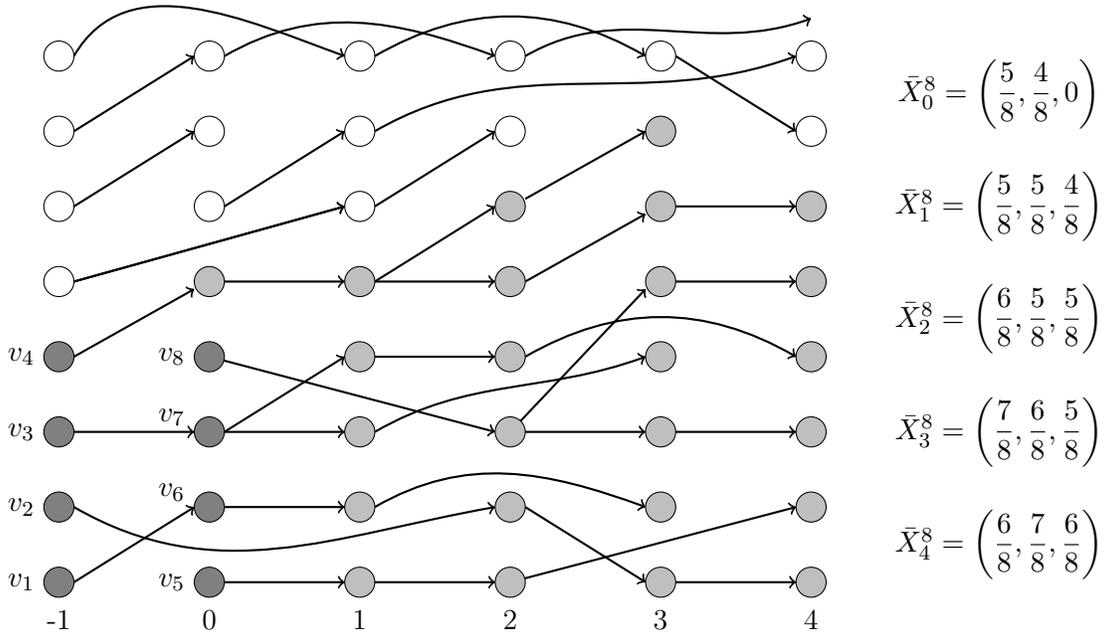
Again for simplicity, we suppose that $\sup\{i\ge1:x_i>0\}$ does not depend on $N$.
\begin{proposition}\label{samplingdual}
Fix the parameters $N$, $\mu^N$ {and $\mathcal W^N$} of the seed bank di-graph. 
The processes $\{\bar X^{N}_g\}_{g\geq0}$ and $\{\bar A^{N}_g\}_{g\geq0}$ are sampling duals: for every $g\geq0$, we have $ {\bf E}_{\bar x}[h^0(\bar n, \bar X^{N}_g)]=\E_{\bar n}[h^0(\bar A_g^{N},\bar x)]$ where $h^0(\bar n,\bar x):=\p_{\bar n}(\mathcal A^N_{1}\subset S_0({\bar x}))$.
\end{proposition}
\begin{proof}
Suppose that the ancestral process starts at generation $g+1$ from the sample $S_{g+1}(\bar n)$, as in Definition \ref{AncProc}.
Also suppose that the frequency process starts at generation 0 from the sample $S_{0}(\bar x)$, as in Definition \ref{FreqProc}. 
Introduce the functions
\begin{equation}
h^g(\bar n,\bar x):=\p_{\bar n}(\mathcal A^N_{g+1}\subset S_0({\bar x})).
\end{equation}
We can write $h^g(\bar n,\bar x)$ in terms of the forward process by conditioning as follows.
\begin{eqnarray*}
h^g(\bar n,\bar x)&=&\sum_{\bar y\in([N]/N)^{m_N}}h^0(\bar n,\bar y){\bf P}_{\bar x}(\bar X_g^{N}=\bar y)\\&=&{\bf E}_{\bar x}[h^0(\bar n, \bar X^{N}_g)].
\end{eqnarray*}
At this point it should be clear that we can also condition according to the backward process.
\begin{eqnarray*}
h^g(\bar n,\bar x)&=&\sum_{\bar m\in [N]^{m_N}}h^0(\bar m, \bar x)\p_{\bar n}(\bar A_g^{N}=\bar m)\\&=&\E_{\bar n}[h^0(\bar A_g^{N},\bar x)].
\end{eqnarray*}
This implies that for all $\bar x\in ([N]/N)^{m_N}$, $\bar n\in [N]^{m_N}$ and $g\geq1$,
$$
\E_{\bar n}[h^0(\bar A_g^{N},\bar x)]={\bf E}_{\bar x}[h^0(\bar n,\bar X^{N}_g)].
$$
\end{proof}


The sampling duality provides a relation between the forward and the ancestral processes. However, in this case the relation is not a moment duality. 
It is possible to write explicitly this sampling duality and to use it, but we will rather use the less natural window process which has the advantage of being precisely the moment dual of the forward process.

\begin{proposition}\label{seedduality}
Fix the parameters $N$, $\mu^N$ {and $\mathcal W^N$} of the seed bank di-graph. 
The window process $\{\bar B^N_g\}_{g\geq0}$ and the forward frequency process $\{\bar X^N_g\}_{g\geq0}$ are moment duals.
\end{proposition}
\begin{proof}
 We will construct a sampling duality that is exactly moment duality i.e duality with respect to the function $H:\N^{m_N}\times [0,1]^{m_N}\mapsto [0,1]$, 
\begin{equation}
H(\bar n,\bar x)=\prod_{i=1}^{m_N} x_i^{n_i}
\end{equation}
 Fix $\bar n\in [N]^{m_N}$ and $\bar x\in ([N]/N)^{m_N}$, and set the samples $S_{0}(\bar n)$ and $S_0(\bar x)$ as in Definition \ref{AncProc}  and Definition \ref{FreqProc} (with $g_0=0$). 
 Observe that $ S_0(\bar n)=\{v=(1-i,U_{j,i}), i=1,\ldots,m_N,j=1,\ldots,n_i\}$ where the $U_{j,i}$'s form a family of independent uniformly distributed random variables with values in $[N]$. 
 Then, we have
 \begin{align*}
\tilde h(\bar n,\bar x):=\p( S_{0}(\bar n)\subset S_0({\bar x}))=\prod_{i=1}^{m_N}\prod_{j=1}^{n_i}\p((1-i,U_{j,i})\in S_0({\bar x}))=\prod_{i=1}^{m_N}\prod_{j=1}^{n_i}x_i=H(\bar n,\bar x).
\end{align*}

Now we prove sampling duality with respect to this function.  
As in the proof of Proposition \ref{samplingdual}, condition on $\bar X^N_g$ to obtain that $\p_{\bar n}(\mathcal B^N_g\subset S_0({\bar x}))={\bf E}{_{\bar{x}}}[\tilde h(\bar n, \bar X^N_g )]$ and condition on $\bar B^N_g$ to obtain that $\p_{\bar{n}}(\mathcal B^N_g{\subset} S_0({\bar x}))=\E{_{\bar{n}}}[\tilde h(\bar B^N_g, \bar x )]$.
\end{proof}

Now we are able to state an analogue of Theorem \ref{cv1KKL} for the dual process, using the moment duality.
\begin{theorem}[Convergence of the forward frequency process]\label{Forward}
Assume that $m_N\leq m<\infty$ for all $N\in\N$. Fix $\{\mathcal{W}^N\}_{N\geq1}$ and $\{\mu^N\}_{N\geq1}$ (and the associated stationary distribution $\nu^N$) such that either the assumptions of Theorem \ref{cv1KKL} hold or the assumptions of Theorem \ref{cv2KKL} hold. { Suppose that $\nu^N$ converges to a measure $\nu$ on $[m]$ as $N\rightarrow\infty$}.
Let $\{\bar X^N\}_{N\geq1}$ be the sequence of frequency processes with parameters $N$, $\mathcal{W}^N$ and $\mu^N$ and starting condition $\bar X^N_{0}=(\lfloor Nx_1\rfloor/N,\dots,\lfloor Nx_{m}\rfloor/N)$ for some $\bar x\in[0,1]^{m}$.\\
i) Under the assumptions of Theorem \ref{cv1KKL},
$$
\lim_{N\rightarrow \infty}\{\bar X_{\lfloor t\beta_N^2 /c_N\rfloor}^N\}_{t\geq0}= \{\bar X_{t}\}_{t\geq0}
$$
in the finite dimension sense,
where $\bar X_t$ is a vector with $m$ identical coordinates $X_t$ such that 
 $X_0=x_0=\sum_{i=1}^m \nu(i)x_i$ a.s.,  and $\{X_t\}_{t\geq0}$ is the Wright-Fisher diffusion (dual of $\{N_{t}^{K}\}_{t\geq0}$).\\
ii) Under the assumptions of Theorem \ref{cv2KKL}, 
$$
\lim_{N\rightarrow \infty}\{\bar X_{\lfloor t /c_N\rfloor}^N\}_{t\geq0}= \{\bar X_{t}\}_{t\geq0}
$$
in the finite dimension sense,
where $\bar X_t$ is a vector with $m$ identical coordinates $X_t$ such that 
 $X_0=x_0=\sum_{i=1}^m \nu(i)x_i$ a.s.,  and 
$\{X_t\}_{t\geq0}$ is the moment dual of $\{N_{t}^{\Xi^\beta}\}_{t\geq0}$.
\end{theorem}
\begin{remark}
{It is surprising at first that the components of the limit are identical. However, this becomes intuitive when observing that the support of  $\mu^N$  vanishes on the limiting time scale.  Theorem  \ref{Forward} uses the more restrictive assumption that the support for $\{\mu^N\}_{N\geq1}$  is bounded. This seems to be more than a technical assumption, because it is hard to believe that an asymptotically infinite dimensional sequence of processes would converge to an infinite dimensional process with all entries being equal (even if  the support of  $\mu^N$  vanishes on the limiting time scale). A natural question is, what is the limit in this more general scenario? An interesting framework to study this question, in which there are results in the literature, is taking $\mu^N:=\mu$ such that $\mu(\{j:j>i\})\sim i^{-\alpha}$ see \cite{BGKS}. If $\alpha>1$, Theorem \ref{Forward} should essentially apply, but the convergence should hold only for finitely (but arbitrarily big) many dimensions simultaneously. If $\alpha\in(0,1/2)$ the limit, if it exists, should be related to fractional Brownian motion (see \cite{HS} and \cite{WI}). The frequency process in the case $\alpha\in(0,1/2)$ is believed to exist and was named by Blath and Span\`o the Fractional Wright Fisher Diffusion. If one takes $\mu^N=1/N\delta_1+Geo(1/N)$ a modification of Theorem \ref{Forward} leads to a criteria for convergence to $\Xi$-seed bank diffusions, this is current work of the authors. The seed bank diffusion was introduced in \cite{BGKW} and can be thought as a delayed stochastic differential equation \cite{BBGW19} (see also \cite{BHN}).}
\end{remark}
\begin{proof}
We only write the details for case {\it i)}, case {\it ii)} follows identically.
 The proof is a consequence of Proposition \ref{seedduality}, Theorem \ref{cv1KKL}  and the moment problem. Let us abuse the notation and write $\bar X^N_0=\bar x$ for every $N$. 
 
First, let us clarify the role of $x_0$. Recall that the process $\{X_t\}_{t\geq0}$  is a martingale. In particular, its expectation remains constant. We claim that for every $i\in [m]$, $\lim_{N\rightarrow \infty}{\bf E}_{\bar x}[X_{\lfloor t\beta_N^2 /c_N\rfloor}^{N,i}]=x_0$. To see this we use duality and convergence to stationarity of a single dual particle. 

\begin{eqnarray}\label{cool}
\lim_{N\rightarrow \infty}{\bf E}_{\bar{x}}[X_{\lfloor t\beta_N^2 /c_N\rfloor}^{N,i}]=\lim_{N\rightarrow\infty}\E_{e_i}\left[ \prod_{j=1}^{m} {x_j}^{B^{N,j}_{\lfloor t { \beta_N^2/c_N}\rfloor}}\right] = \sum_{i=1}^m x_i\nu(i)=x_0.
\end{eqnarray}
 The first  equality comes from duality. For the second equality, {recall that in both Theorem \ref{cv1KKL} and Theorem \ref{cv2KKL} we suppose that $\beta_N^2/c_N\rightarrow \infty$ and thus that the process is in stationarity in the limit.} The third equality follows from the fact that there is only one positive entry of the unitary vector $\bar{B}^N_{\lfloor t  \beta_N^2/c_N\rfloor}$  and that the position of the entry with the one is $\nu$-distributed in the limit.
 
 Now let us study the limiting behavior of one coordinate.
 Let $n\geq1$.
\begin{eqnarray*}
\lim_{N\rightarrow\infty}{\bf E}_{\bar{x}}[(X_{\lfloor t\beta_N^2/c_N\rfloor}^{N,1})^{{n}}]&=&\lim_{N\rightarrow\infty}
{\bf E}_{\bar x}[H({n}.e_1,\bar X_{\lfloor t\beta_N^2/c_N\rfloor}^{N})]\\
&=&\lim_{N\rightarrow\infty}
{\E}_{ n.e_1}[H(\bar{B}^{N}_{\lfloor t\beta_N^2/c_N\rfloor},\bar {x})]\\
&=&{\E}_{ n.e_1}[H({V}^{t,K},\bar {x})]\\
&=&\E_{ n}[x_0^{N_t^K}]\\
&=&{\bf E}_{x_0}[X_{t}^{n}].
\end{eqnarray*}
The third equality follows from \eqref{multinomialdist} and in the fourth equality we used the same argument as for \eqref{cool}. 
This proves that all the moments of $X_{ \lfloor t\beta_N^2 /c_N\rfloor}^{N,1}$ converge to the moments of the Wright-Fisher diffusion.

  Finally, we check that in the limit all the coordinates of $\bar X_{\lfloor t \beta_N^2 /c_N\rfloor}^N$ must take the same value. To do this we will calculate the square of the difference of two arbitrary coordinates. Let $i,j\in [m]$. With the same arguments as before,
\begin{eqnarray*}
\lim_{N\rightarrow\infty}{\bf E}_{\bar x}[(X_{\lfloor t\beta_N^2/c_N\rfloor}^{N,i}-X_{\lfloor  t\beta_N^2/c_N\rfloor}^{N,j})^2]&=&\lim_{N\rightarrow\infty}\left[{\bf E}_{\bar x}[(X_{\lfloor t\beta_N^2/c_N\rfloor}^{N,i})^2]+{\bf E}_{\bar x}[(X_{\lfloor t\beta_N^2/c_N\rfloor}^{N,j})^2]\right.\\
&&\left.-2{\bf E}_{\bar x}[X_{\lfloor t\beta_N^2/c_N\rfloor}^{N,i}X_{\lfloor t\beta_N^2/c_N\rfloor}^{N,j}]\right]\\
&=&\lim_{N\rightarrow\infty}\left[{\bf E}_{\bar x}[H(2e_i,\bar X_{\lfloor t\beta_N^2/c_N\rfloor}^{N})]+{\bf E}_{\bar x}[H(2e_j,\bar X_{\lfloor t\beta_N^2/c_N\rfloor}^{N})]\right.\\
&&\left.-2{\bf E}_{\bar x}[H(e_i+e_j,\bar X_{\lfloor t\beta_N^2/c_N\rfloor}^{N})]\right]\\
&=&\lim_{N\rightarrow\infty}\left[\E_{2e_i}[H(\bar B_{\lfloor t\beta_N^2/c_N\rfloor}^{N},\bar x)]+\E_{2e_j}[H(\bar B_{\lfloor t\beta_N^2/c_N\rfloor}^{N},\bar x)]\right.\\
&&\left.-2\E_{e_i+e_j}[H(\bar B_{\lfloor t\beta_N^2/c_N\rfloor}^{N},\bar x)]\right]\\
&=&\E_{2e_i}[H(V^{t,K},\bar x)]+\E_{2e_j}[H(V^{t,K},\bar x)]\\
&&-2\E_{e_i+e_j}[H({V}^{t,K},\bar x)]\\
&=&0.
\end{eqnarray*}
This ends the proof.

\end{proof}

\section{Extended seed bank di-graph with mutations}\label{S4}

It shall be interesting to study a variant of the model where mutations are added.
There is a classical duality relation between the Wright-Fisher diffusion with mutations
\begin{equation}
dX_t=(u_1(1-X_t)-u_2X_t){dt}+\sqrt{X_t(1-X_t)}dB_t.
\end{equation}
and the block counting process of a Kingman coalescent with freezing, where every lineage can disappear at rate $u_1+u_2$. This relation was established in  \cite{Etheridge_2009} (see also \cite{Griffiths} and the seminal work of \cite{EG93}) and was generalized in \cite{EGT} to $\Lambda$-coalescents with freezing.

These works motivate that we modify the seed bank di-graph to include mutations (Definition \ref{ExtSeedgraph}). In order to observe genealogies with freezing in the model we consider that mutations come from a separate source and not from a reproduction event. This will let us establish a duality relation even in the finite population case (Proposition \ref{seedduality2}) between a modification of the window process and a modification of the forward frequency process, and hence,  a duality relation in the limit generalizing known results (Theorem \ref{Forwardmutation}). 
The limit genealogical processes obtained in this relation are described in Theorems \ref{convergeKM} and \ref{cv2mutation}.  They are general coalescents but with a freezing component. We give a formal description of them that is more convenient for our setting than the common definition in Definition \ref{defbcf}.
Most of the proofs of this section are very similar to those of Sections \ref{S2} and \ref{ffp}. We emphasize the modifications and the intuitions in the proof of Theorem \ref{convergeKM} and enunciate the other results without the proofs, {leaving} them to the reader.
For sake of simplicity, most of the notations of this section will be identical to those of the previous ones when the objects describe the same concepts, although their definitions are slightly modified.

\begin{definition}[The extended seed bank random di-graph]\label{ExtSeedgraph}
Set $N, \mu^N$ and $\mathcal {W}^N$ as in Definition \ref{SeedBankgraph}, the vertex set $V^N$, and also the random variables $\{\bar W^N_g\}_{g\in\mathbb Z}$, where $\bar W^N_g=\{W_{v}^{N}\}_{v\in V^N_g}$, $\{J^N_v\}_{v\in V^N}$ and $\{U^N_v\}_{v\in V^N}$.
Fix two more parameters $u_{N,1},u_{N,2}\geq 0$ such that $u_{N,0}:=1-u_{N,1}-u_{N,2}\geq 0$, and let $\{K^N_v\}_{v\in V^N}$ be a sequence of independent random variables with state space $\{0,1,2\}$ such that $\mathbb{P}(K^N_v=i)=u_{N,i}$. 
Consider the extended vertex set $V^N\cup\{\Delta_1\}\cup\{\Delta_2\}$ and the random function
$F: V^N\mapsto V^N\cup\{\Delta_1\}\cup\{\Delta_2\}$ defined by the rule
\begin{equation}\label{coupling3}
F(v)=\left\{ \begin{array}{ll}
(g(v)-J^N_v,U^N_v) & \textrm{ if } K^N_v=0\\
\Delta_i & \textrm{ if } K^N_v=i\in\{1,2\}.
\end{array}
\right .
\end{equation}
 Let $E^N$ be the set of directed edges $$E^N=\{\{v,F(v)\}, \,\text{ for all }v\in V^N\}.$$
  {The extended seed bank random di-graph} with parameters  $N,$ $\mu^N$, $\mathcal{W}^N$ and $u_{N,1},u_{N,2}$ is given by $ G^N:=( V^N\cup\{\Delta_1\}\cup\{\Delta_2\},E^N)$. 
\end{definition}

In this extended version of the graph, we can define a modification of the original window process.

\begin{definition}[The window process with mutations]\label{winmut}
Fix a generation $g_0$, and $S_{g_0} \subset\cup_{i=1}^{m_N}  V^N_{g_0+1-i}$.
The window process with mutations  is the chain $\{\bar B^N_g,D^N_g\}_{g\geq0}$ where $\{\bar B^N_g\}_{g\geq0}$ is the window process introduced in Definition \ref{WindProc} (but associated to the extended di-graph) and $\{D^N_g\}_{g\geq0}$ is the process counting the cumulate number of lineages connecting with $\Delta_1$ or $\Delta_2$ under the rule $F$, that is $$D^N_g=|\{(v,v')\in  E^N, g(v)\geq g_0-g {, v \text{ is an ancestor of some individual of } S_{g_0}} \text{ and } v' \in\{\Delta_1,\Delta_2\}\}|$$
We denote by $\p_{\bar n}$ the law of the window process starting from $S_{g_0}(\bar n)$ and with $D_0^N=0$.
\end{definition}
\begin{example}
Modify Figure \ref{FigAncestral} such that individual ${(-2,7)}$ is connected to $\Delta_1$. The window process with mutations has the following values: 
${\{\bar{B}^8_0, D^8_0\}}=\lbrace {(4,1)},0\rbrace$, $\{\bar{B}^8_1, D^8_1\}=\lbrace {(2,3)},0\rbrace$, $\{\bar{B}^8_2, D^8_2\}=\lbrace {(5,0)},0\rbrace$, $\{\bar{B}^8_3, D^8_3\}=\lbrace {(1,1)},1\rbrace$,  $\{\bar{B}^8_4, D^8_4\}=\lbrace {(2,0)},1\rbrace$, $\{\bar{B}^8_5, D^8_5\}=\lbrace {(1,0)},1\rbrace$.
\end{example}
The state $\Delta_1$ can be seen as the source of type $a$ mutations and the state $\Delta_2$ as the source of type $A$ mutations. So we can slightly modify Definition \ref{FreqProc} to obtain the new forward frequency process.

\begin{definition}[The frequency process with mutations]\label{forward2}
Fix a generation $g_0$ and an initial sample $S_{g_0}\subset\cup_{i=1}^{m_N}  V^N_{g_0+1-i}$, that we call the type $A$ individuals. Hence, $\cup_{i=1}^{m_N}  V^N_{g_0+1-i}\backslash S_{g_0}$  is the set of type $a$ individuals. 
 For $g\geq0$, set (omitting again the dependence to $S_{g_0}$) 
$$
X^{N,i}_g=\frac{1}{N}|\{v\in V^N_{g_0+g+1-i}: v\text{ is not connected to }\Delta_1\text{ nor to }u\text{ for some }u\in \cup_{i=1}^{m_N}  V^N_{g_0+1-i}\backslash S_{g_0} \}|.
$$
Then, define the process of the neutral frequency of type $A$ individuals $\{\bar X^{N}_g, \theta^N\}_{g\geq0}$, where \\$\theta^N=u_{N,1}/(u_{N,1}+u_{N,2})$ and
$$
\bar X^{N}_g=(X_g^{N,1},\dots,X_g^{N,m_N}).
$$
Set a vector $\bar x=(x_1, \dots, x_{m_N})\in([N]/N)^{m_N}$. In the sequel, we suppose that the forward frequency process starts from a fraction $x_1$ of generation 0, a fraction $x_2$ of generation $-1$, and so on. We denote this sample by $S_{0}(\bar x)=\cup_{i=1}^{m_N}\cup_{k=1}^{ x_iN}\{(1-i,k)\}$ and we denote the law of the frequency process starting from this configuration  by ${\bf P}_{\bar x}$.
\end{definition}

We obtain a duality result, which is the analogue of the moment duality obtained in Proposition \ref{seedduality}, and that can easily be proved adapting the proofs of Section \ref{ffp}.
It is inspired by  \cite{Etheridge_2009} and \cite{Griffiths}.

\begin{proposition}\label{seedduality2}
Fix $N$, $\mu^N$, $\mathcal W^N$ and $u_{N,1},u_{N,2}$. The window process $\{\bar B^N_g,\bar D^N_g\}_{g\geq0}$ and the frequency process with mutations $\{\bar X^N_g,\theta^N\}_{g\geq0}$ are moment duals, in the sense that
for any $g\geq0$,
$$
{{\bf E}_{\bar x}\left[ \prod_{i=1}^{m_N} (X_g^{N,i})^{n_i}\right] =\E_{\bar n}\left[ (\theta^N)^{D_g^N}\prod_{i=1}^{m_N} x_i^{B_g^{N,i}}\right]}.
$$
\end{proposition}

Let us recall the concept of coalescent with freezing.
Usually this model describes an exchangeable coalescent where, apart from participating to merging events, lineages can disappear (or freeze) at a constant rate, independent of each other.
It appears in general population models with mutations in \cite{EGT}.
Its Kingman version is also crucial in the infinite alleles model, where Ewens' sampling formula is established \cite{Ewens_1972}, but also to provide asymptotics to the seed bank (or peripatric) coalescent \cite{GPS}.

As could be anticipated from Definition \ref{winmut}, we will modify a bit this object keeping track of the whole genealogy and  of the frozen lineages separately. We only need to consider its block counting process for our purpose

\begin{definition}[The block counting process of a coalescent with freezing]\label{defbcf}
The block counting process of a $\Xi$-coalescent with freezing parameter $u$ is the process $\{M_t^\Xi, D_t\}_{t\geq0}$ where $M_t^\Xi$ counts the number of (remaining) blocks at time $t$ and $D_t$ counts the cumulate number of frozen blocks until time $t$.
This process jumps from $(n,m)$ to $(n-k,m)$ when a coalescent event occurs (for some $k\in[n-1]$) and to $(n-1,m+1)$ when a  freezing event occurs.
In the special case of a Kingman coalescent with freezing, we use the notation $\{M_t^K, D_t\}_{t\geq0}$.
\end{definition} 

We obtain the following analogue of Theorem \ref{cv1KKL}.
\begin{theorem}[Convergence of the window process I: Kingman limit]\label{convergeKM}
Consider an extended seed bank di-graph with parameters $N, \mu^N, \mathcal W^N$ and $u_{N,1}, u_{N,2}$.
Assume that conditions of Theorem \ref{cv1KKL} hold, plus the following
 $$u_{N,0}\to u_0, \quad\frac{u_{N,1}\beta_N}{c_N}\to u_1, \quad \frac{u_{N,2}\beta_N}{c_N}\to u_2,$$
where $u_0\in(0,1]$ and $u_1,u_2>0$.
 Consider the window process with mutations starting at $\bar B_{0}^N=\bar n$ and $D^N_0=0$ for all $N\in\N$ big enough. Then,
\begin{equation}\label{conv1}
\lim_{N\rightarrow\infty}\{|\bar B_{\lfloor t \beta_N^2/c_Nu_{N,0}^2\rfloor}^N|, D_{\lfloor t \beta_N^2/c_Nu_{N,0}^2\rfloor}^N\}_{t\geq0}= \{M_t^K, D_t\}_{t\geq0}
\end{equation}
in the finite dimension sense,
 where $\{M_t^K, D_t\}_{t\geq0}$ is the block counting process of a Kingman coalescent with freezing parameter $(u_1+u_2)/u_{0}^2$. 
\end{theorem}

Before proving this result, we need to modify the coupling particle system introduced in section \ref{ParticleEasy}.
Let $ Y_g^N=(R_g^N,L_g^N)$ define a Markov chain with state space $(\mathbb{N}\times [N])\cup\{\Delta_1,\Delta_2\}$ and transition probabilities, conditional on the weights $\bar{W}_g^N$,
 $$
\p\big(Y^N_g=(i,k)|\{\bar{W}^N_{g}\}_g,Y^N_{g-1}=(1,j)\big)={W}_{(g_0-g+1-i,k)}^{N}\mu^N(i){u_{N,0}}
$$
for every $i\geq1$, $k\in [N]$,  
$$
\p\big(Y^N_g=\Delta_i |\{\bar{W}^N_{g}\}_g,Y^N_{g-1}=(1,j)\big)=u_{N,i}
$$
for $i\in\{1,2\}$,
and
$$
\p\big(Y^N_g=(i,k)|\{\bar{W}^N_{g}\}_g,Y^N_{g-1}=(i+1,j)\big)={W}_{(g_0-g,k)}^{N}
$$
for every $i\geq1$ and $k\in [N]$.
We enunciate the coupling result without proof.

\begin{proposition}\label{mutation}
Set $n=\sum n_i$ to be the total size of the initial sample.
For every $g\geq0$, consider $n$ {(conditional on $\{\bar{W}^N_{g}\}_g$)} independent realizations  of $Y^N_g$, that we call $Y^{N,j}_g=(R^{N,j}_g,L^{N,j}_g)$ for $1\leq j\leq n$, and set
 $$\gamma^{N,j}=\inf\{g\geq 1: Y^{N,j}_g\in\{\Delta_1,\Delta_2\} \}.$$  
 Recall $\{\sigma^{N,j}\}_{j=1}^n$ from Proposition \ref{ParticleEasy}.
For all $i\geq1$, set $\sum_{j=1}^n1_{\{R^{N,j}_0=i\}}=B^{N,i}_0$.   
Then, the $i$-th component $B_g^{N,i}$ of the random vector $\bar B_g^N$ is equal in distribution to
$\sum_{j=1}^n1_{\{R^{N,j}_g=i\}}1_{\{\sigma^{N,j}\ge g\}}1_{\{\gamma^{N,j}\ge g\}}$, for all $g\geq0$ and $D_g^{N}$ is equal in distribution to $\sum_{k=1}^g\sum_{j=1}^n1_{\{\gamma^{N,j}=k\}}1_{\{\sigma^{N,j}\ge k\}}$.
\end{proposition} 
We are now ready to prove Theorem \ref{convergeKM}.
\begin{proof}[Proof of Theorem \ref{convergeKM}]
As in the proof of Theorem \ref{cv1KKL}, recall the invariant measure $\nu^N$ defined in \eqref{nuN}, define $\underline{Y}^{N,j}_g=(\underline{R}^{N,j}_g,L^{N,j}_g)$ where $\{\underline{R}^{N,j}_g\}_{g\geq0}$ is a sequence of independent $\nu^N$-distributed random variables, and let $\underline{\sigma}^{N,1}=\infty$ and for $2\le j\le n$,
 $$\underline{\sigma}^{N,j}=\inf\{g\geq1: \underline{Y}^{N,j}_g=\underline{Y}^{N,j'}_g=(1,k),\text{ for some }j'<j \text{ such that } \underline{\sigma}^{N,j'}>g, k\in[N]\}.$$
To couple the variables $(\gamma^{N,j})_{j=1}^n$, observe that each of them  can be associated to a geometric r.v. of parameter $u_{N,1}+u_{N,2}$. More precisely, let $(G^{N,1},\dots, G^{N,n})$ be a family of independent geometric r.v.s of parameter $u_{N,1}+u_{N,2}$. Then we define

 $$\underline{\gamma}^{N,j}=1+\inf\{g\geq1:\sum_{k=0}^g1_{\{\underline{R}^{N,j}_k=1\}}=G^{N,j}\}.$$
Hence, we define an artificial window process with mutations 
$
\{\bar Z^N_g,{I}_g^N\}_{g\geq0}
$
where $\bar Z^N_g=( Z^{N,1}_g, \dots,  Z^{N,m_N}_g)$ has coordinates given by
$$
Z^{N,i}_g=\sum_{j=1}^n1_{\{\underline{R}^{N,j}_g=i\}}1_{\{\underline{\sigma}^{N,j}\ge g\}}1_{\{\underline{\gamma}^{N,j}\ge g\}}
$$
and
$$
I^{N}_g=\sum_{k=1}^g\sum_{j=1}^n1_{\{\underline\gamma^{N,j}=k\}}1_{\{\underline\sigma^{N,j}\ge k\}}.
$$
The process $\{|\bar Z^{N}_g|, I_g^N\}_{g\geq0}$ is Markovian. 

 First, we calculate the generator of $\{(|\bar Z^{N}_g|, I^N_g)\}_{g\geq0}$ in order to discover its scaling limit. Let $f:\N\times\N\rightarrow \R^2$ be a bounded function. 
Then
\begin{align*}
\mathcal G^Nf(n,m)&=\E[f(|\bar Z^{N}_1|, I^N_1)-f(n,m)]\\
&=\p((|\bar Z^{N}_1|, I_1^N)=(n-1,m))[f(n-1,m)-f(n,m)]\\
&\hspace{0.4cm}+\p((|\bar Z^{N}_1|, I_1^N)=(n-1,m+1))[f(n-1,m+1)-f(n,m)]\\
&\hspace{0.4cm}+O(\p(|\bar Z^{N}_1|=n-2))\\
&=\binom{n}{2}c_N{(\nu^N(1)u_{N,0})^2}[f(n-1,m)-f(n,m)]\\
&\hspace{0.4cm}+n \nu^N(1)\left({u_{N,1}}+ {u_{N,2}}\right)[f(n-1,m+1)-f(n,m)]\\
&\hspace{0.4cm}+O\left( {d_N}{(\nu^N(1)u_{N,0})^3}+\left(\nu^N(1)\left({u_{N,1}}+ {u_{N,2}}\right)\right)^2+c_N{(\nu^N(1))^3u_{N,0}^2}\left({u_{N,1}}+ {u_{N,2}}\right)\right)\\
&=\frac{c_Nu_{N,0}^2}{\beta_N^2}\left[\binom{n}{2}[f(n-1,m)-f(n,m)]\right. \\
&\hspace{0.4cm}+\left. n \left(\frac{u_{N,1}\beta_N}{c_Nu_{N,0}^2}+ \frac{u_{N,2}
\beta_N}{c_Nu_{N,0}^2}\right)[f(n-1,m+1)-f(n,m)]\right.\\
&\left.\hspace{0.4cm}+O\left( \frac{d_N}{\beta_Nc_N}{u_{N,0}}+\frac{c_N}{\beta_N^2}\left(\frac{u_{N,1}\beta_N}{c_Nu_{N,0}}+ \frac{u_{N,2}\beta_N}{c_Nu_{N,0}}\right)^2+\frac{c_N}{\beta_N^2}\left(\frac{u_{N,1}\beta_N}{c_N}+ \frac{u_{N,2}\beta_N}{c_N}\right)\right)\right].
\end{align*}
 So we conclude that
\begin{equation}
\{|\bar Z^{N}_{\lfloor t\beta_N^2 /c_Nu_{N,0}^2\rfloor}|, I^N_{\lfloor t\beta_N^2 /c_Nu_{N,0}^2\rfloor}\}_{t\geq0}\Rightarrow \{M_{t}^K, D_t\}_{t\geq0},
\end{equation}
the block counting process of a Kingman coalescent with freezing parameter $(u_1+u_2)/u_{0}^2$. 

To see that $\{|\bar B^{N}_{\lfloor t\beta_N^2 /c_Nu_{N,o}^2\rfloor}|, D^N_{\lfloor t\beta_N^2 /c_Nu_{N,0}^2\rfloor}\}_{t\geq0}\Rightarrow \{M^K_{t}, D_t\}_{t\geq0}$,  we couple $\{|\bar Z^{N}_{\lfloor t\beta_N^2 /c_N \rfloor}|\}_{t\geq0}$\\ and $\{|\bar B^{N}_{\lfloor t\beta_N^2/c_N\rfloor}|\}_{t\geq0}$ mimicking the proof of Theorem \ref{cv1KKL}  to show that the same limit is true for the rescaled window process (since $u_{N,0}\to u_0$ with no scaling, we can suppose that $u_{N,0}=1$). 
{In this case, we still denote, for any $p,q\in[n]$, $\rho^{N,p,q}_k=\inf\{g>\rho^{N,p,q}_{k-1}:L^{N,p}_g=L^{N,q}_g\}$ (with $\rho_0^{N,p,q}=0$), but now the potential jump times are $\rho_i=\inf\{g>\rho_{i-1}:g=\rho^{N,p,q}_k$ for some $p,q\in[n]$ and some $k\in\N$ { or} $g=G^{N,j}$ for some $j\in[n] \}$ (with $\rho_0{=0}$).
}
Recall that $(G^{N,1},\dots, G^{N,n})$ is a family of independent geometric r.v.s of parameter $u_{N,1}+u_{N,2}$.
The probability that the coupling is successful is

\begin{eqnarray*}
{p}_N&:=&\inf_{\bar n \in [N]^{m_N}} \p_{\bar n}((\underline{R}^{N,1}_{\hat{\rho}_1},\dots,\underline{R}^{N,m_N}_{\hat{\rho}_1})=(R^{N,1}_{\hat{\rho}_1},\dots,R^{N,m_N}_{\hat{\rho}_1}))\\
&=&1-\sup_{\bar n \in [N]^{m_N}}\p_{\bar n}((\underline{R}^{N,1}_{\hat{\rho}_1},\dots,\underline{R}^{N,m_N}_{\hat{\rho}_1})\neq(R^{N,1}_{\hat{\rho}_1},\dots,R^{N,m_N}_{\hat{\rho}_1}))\\
&=&1-\sup_{\bar n \in [N]^{m_N}}\| \p_{\bar n}((R^{N,1}_{\hat{\rho}_1},\dots,R^{N,m_N}_{\hat{\rho}_1})=\cdot)-(\nu^N)^{\otimes m_N}(\cdot)\|_{TV}
\end{eqnarray*}
where $\p_{\bar n}$ stands for the law of $\{R^{N,1}_{g},\dots,R^{N,m_N}_{g}\}_{g\geq0}$ (or $\{\underline R^{N,1}_{g},\dots,\underline R^{N,m_N}_{g}\}_{g\geq0}$)  starting at the state $\bar n\in [N]^{m_N}$ and where Proposition 4.7 in \cite{Mixing} is used for the last equality.
To prove that ${p}_N\to1$ when $N\to\infty$, take $\varepsilon>0$ such that $N^\varepsilon\tau_N c_N\rightarrow 0$ (and thus $\tau_NN^{\varepsilon}(u_{N,1}+u_{N,2})\rightarrow0$) . 
The condition $\mu^N(1)>0$ implies that, for any $i\geq1$,  the processes $\{R_g^{N,i}\}_{g\geq0}$ are irreducible. So, by Theorem 4.9 in \cite{Mixing}, we have

\begin{equation*}\label{coumpl1}
||\p_{\bar n}((R^{N,1}_{N^\varepsilon\tau_N },\dots,R^{N,m_N}_{N^\varepsilon\tau_N })=\cdot)-(\nu^N)^{\otimes m_N}(\cdot)||_{TV}<(1/4)^{N^\varepsilon}.
\end{equation*}
Then observe that, stochastically, ${\rho}_1\geq {\Gamma}^N$ where ${\Gamma}^N$ is a geometric random variable of parameter\\ ${n^2c_N+n(u_{N,1}+u_{N,2}) }$ and thus $\p({\rho}_1\leq N^\varepsilon\tau_N )\leq \p({\Gamma}^N\leq N^\varepsilon\tau_N )\rightarrow 0$.

Let $T^N_1=\inf\{i\geq1: |\bar Z^{N}_{\rho_i}|=1\}$. 
The jump times are reduced compared to those of the proof of Theorem \ref{cv1KKL}, so the inequality \eqref{trucmuche} still holds and we obtain \eqref{conv1} with similar arguments
\end{proof}

A multiple merger version of this result is also obtained. Note that in this case the hypothesis that $\beta_N\to\beta<\infty$ involves that $u_{N,0}\to1$.

\begin{theorem}[Convergence of the window process II: $\Xi$ limit]\label{cv2mutation}
Fix $\{\mu^N\}_{N\in\mathbb{N}}$ such that  $\beta_N=\E[J^N_v]\to\beta<\infty$ and fix the distribution $\mathcal{W}^N$. Assume that the ancestral process of a Cannings model driven by $\mathcal{W}^N$, that we denote by $\{C^N_g\}_{g\geq0}$ is such that
$$
\lim_{N\rightarrow\infty}\{C^N_{\lfloor t/c_N \rfloor}\}_{t\geq0}= \{N_t^\Xi\}_{t\geq0}
$$
where $\{N_t^{\Xi}\}_{t\geq0}$ stands for the block counting process of a $\Xi$-coalescent.
Assume that
 $$\frac{u_{N,1}}{\beta_Nc_N}\to u_1, \quad \frac{u_{N,2}}{\beta_Nc_N}\to u_2,$$
where $u_1,u_2>0$.
 Consider the window process with mutations starting at $\bar B_{0}^N=\bar n$ and $D^N_0=0$ for all $N\in\N$ big enough. Then,\begin{equation}\label{convXimut}
\lim_{N\rightarrow\infty}\{|\bar B_{\lfloor t/c_N\rfloor}^N|,D^N_{\lfloor t/c_N\rfloor}\}_{t\geq0}= \{M_{t}^{\Xi^\beta},D_t\}_{t\geq0}.
\end{equation}
in the finite dimension sense,
 where $\{M_t^{\Xi^\beta}, D_t\}_{t\geq0}$ is the block counting process of a ${\Xi^\beta}$-coalescent with freezing parameter $u_1+u_2$. 

\end{theorem}
Finally we obtain a convergence result for the frequency processes.
Because of the two coordinate notation of the block counting process of a coalescent with freezing, its moment dual is now written as $\{X_t,\theta_t\}_{t\geq0}$
\begin{theorem}[Convergence of the forward frequency process]\label{Forwardmutation}
Assume that $m_N\leq m<\infty$ for all $N\in\N$. Fix $\{\mathcal{W}^N\}_{N\geq1}$ , $\{\mu^N\}_{N\geq1}$ (and the associated stationary distribution $\nu^N$) and $u_{N,1},u_{N,2}$ such that either the assumptions of Theorem \ref{convergeKM} hold or the assumptions of Theorem \ref{cv2mutation} hold. { Suppose that $\nu^N$ converges to a measure $\nu$ on $[m]$ as $N\rightarrow\infty$}.
Let $\{\bar X^N,\theta^N\}_{N\geq1}$ be the sequence of frequency processes with parameters $N$, $\mathcal{W}^N$, $\mu^N$ and $u_{N,1},u_{N,2}$ and starting condition $\bar X^N_{0}=(\lfloor Nx_1\rfloor/N,\dots,\lfloor Nx_{m}\rfloor/N)$ for some $\bar x\in[0,1]^{m}$. 
i) Under that assumptions of Theorem \ref{convergeKM},
$$
\lim_{N\rightarrow\infty}\{\bar X_{\lfloor t\beta_N^2 /c_N\rfloor}^N\}_{t\geq0}= \{\bar X_{t}\}_{t\geq0}
$$
in the finite dimension sense,
where $\bar X_t$ is a vector with $m$ identical coordinates $X_t$ such that 
 $X_0=x_0=\sum_{i=1}^m \nu(i)x_i$ a.s.,  and $\{X_t,u_1/(u_1+u_2)\}_{t\geq0}$ is the moment dual of $\{M_{t}^{K}, D_t\}_{t\geq0}$.\\
ii) Under the assumptions of Theorem \ref{cv2mutation}, 
$$
\lim_{N\rightarrow\infty}\{\bar X_{\lfloor t /c_N\rfloor}^N\}_{t\geq0}= \{\bar X_{t}\}_{t\geq0}
$$
in the finite dimension sense,
where $\bar X_t$ is a vector with $m$ identical coordinates $X_t$ such that 
 $X_0=x_0=\sum_{i=1}^m \nu(i)x_i$ a.s.,  and 
$\{X_t,u_1/(u_1+u_2)\}_{t\geq0}$ is moment dual of $\{M_{t}^{\Xi^\beta}, D_t\}_{t\geq0}$.
\end{theorem}

{\bf Acknowledgement.}
This project was partially supported by UNAM-PAPIIT grant IN104722 and IN101722. LP would like to thank the Universidad del Mar for the support through project 2IIMA2301.

\bibliography{References}{}
\bibliographystyle{babplain}
\end{document}